\documentclass[11pt,english]{smfart}

\usepackage[T1]{fontenc}
\usepackage[english]{babel}

\usepackage{amssymb,url,xspace,smfthm}
\usepackage{amsfonts}
\usepackage{graphicx}

\newcommand{\BibTeX}{{\scshape Bib}\kern-.08em\TeX}
\newcommand{\T}{\S\kern .15em\relax }
\newcommand{\AMS}{$\mathcal{A}$\kern-.1667em\lower.5ex\hbox
        {$\mathcal{M}$}\kern-.125em$\mathcal{S}$}

\newtheorem{lemma}{Lemma}
\newtheorem{remark}{Remark}
\newtheorem{theorem}{Theorem}
\newtheorem{definition}{Definition}

\newtheorem{proposition}{Proposition}
\title[ADP for a Mean-field Game of Traffic Flow]{Approximate Dynamic Programming for a Mean-field Game of Traffic Flow: \\ Existence and Uniqueness}
\date {\today}
\author{Amoolya Tirumalai}

\address{Department of Electrical and Computer Engineering \& Institute for Systems Research, University of Maryland, 8223 Paint Branch Dr, College Park, MD 20740, USA. }
\email{ast256@umd.edu}

\author{John S. Baras}

\address{Department of Electrical and Computer Engineering \& Institute for Systems Research, University of Maryland, 8223 Paint Branch Dr, College Park, MD 20740, USA. }
\email{baras@umd.edu}

\keywords{traffic flow, approximate dynamic programming, mean-field games, weak solutions, semigroups}

\begin{document}
\def\smfbyname{}

\begin{abstract}
	Highway vehicular traffic is an inherently multi-agent problem. Traffic jams can appear and disappear mysteriously. We develop a method for traffic flow control that is applied at the vehicular level via mean-field games. We begin this work with a microscopic model of vehicles subject to control input, disturbances, noise, and a speed limit. We formulate a discounted-cost infinite-horizon robust mean-field game on the vehicles, and obtain the associated dynamic programming (DP) PDE system. We then perform approximate dynamic programming (ADP) using these equations to obtain a sub-optimal control for the traffic density adaptively. The sub-optimal controls are subject to an ODE-PDE system. We show that the ADP ODE-PDE system has a unique weak solution in a suitable Hilbert space using semigroup and successive approximation methods. We additionally give a numerical simulation, and interpret the results.
\end{abstract}
\thanks{This work was supported in part by the US Office of Naval Research (ONR) Grant No. N00014-17-1-2622. The work herein expresses the views of the authors.}
\maketitle
\tableofcontents

\section{Introduction}

Highway vehicular traffic control is an inherently multi-agent control problem. Optimal control methods for such systems fall victim to the curse of dimensionality as the number of agents grows.

Beginning in the 1950s, models of traffic flow were developed based on macroscopic conservation laws \cite{lighthill1955kinematic, richards1956shock}. These models are generally non-linear hyperbolic PDEs for the spatial vehicle density. Solutions to these PDEs often exhibit shockwave \cite{evans2010partial} phenomena, which correspond to formation of traffic jams.

These models are convenient in some ways. For example, boundary control of PDEs is a rich field which has its basis in control of abstract infinite dimensional systems \cite{fattorini1968boundary, baras1974state}, and has been successfully applied to traffic using reinforcement learning \cite{yu2021reinforcement, belletti2017expert}. These are not control methods for individual vehicles, however. Rather, they specify throughput at on-ramps to a stretch of highway such that the highway does not become congested. This does not fulfill our objective which is:

\textbf{Objective:} Obtain a robust control for individual vehicles on a stretch of road which dissipates congestion on the entire road while preserving ride comfort and increasing vehicle speed.

To attempt to reach our objective, we apply mean-field game theory \cite{lasry2007mean, bensoussan2013mean}. Via a large population limit, mean-field games are a way to approximate a finite-size differential game with a single-agent optimal control problem involving PDEs. Mean-field games have been applied successfully to obtain individual-level controls for traffic in pedestrian and vehicular traffic \cite{chevalier2015micro, tirumalai2021robust, festa2018mean, dogbe2010modeling}.

We do not address the issue of safety and collision avoidance between individual vehicles. An assured safety approach which combines a macroscopic perspective of control with the microscopic can be obtained by first calculating the mean-field control, and then correcting it at the individual vehicle level via control barrier function quadratic-programming \cite{ames2019control}, for example. To our knowledge, such a problem is open. We leave this to subsequent work of ours and others.

\textbf{Summary of problem.} We begin our formulation of the mean-field game with the agent dynamics. In the initial problem description, we have finitely many agents. The agents are bound to a closed track of length $\mathfrak L$. The agents are subject to control input, disturbance, and Gaussian noise. A speed limit and travel direction are imposed as state constraints. Control and disturbance constraints are imposed in addition.

To deal with the state constraints, instead of imposing them as algebraic constraints to our dynamics or on the optimal control problem itself, or considering the dynamics until an exit time, we explicitly include them in the dynamics using the theory of reflected diffusions \cite{ watanabe1971stochastic, tanaka1979stochastic, lions1984stochastic, pilipenko2014introduction} and employ the associated optimal control theory \cite{lions1984optimal, lions1985optimal}. The inclusion of reflecting boundary conditions in mean-field games is briefly mentioned in \cite{lasry2007mean}. Reflecting boundaries are studied in McKean-Vlasov-type mean-field SDEs in \cite{coghi2022mckean}.

We first formulate a finite-size robust stochastic differential game with discounted infinite horizon cost on the constrained-state vehicle dynamics. Then, we formally (as opposed to rigorously) pass to the analogous constrained-state robust infinite-horizon mean-field game. The infinite horizon is chosen so that the time-derivatives of the value function do not enter into the dynamic programming (DP) PDEs. This makes the approximate dynamic programming (ADP) procedure somewhat simpler.

The cost functional we choose in our game accounts for ride comfort, congestion dissipation, and a preference for the highest speed of travel. The problem is an extension of our work in \cite{tirumalai2021robust}. 

We use this mean-field game system to formulate an ADP  ODE-PDE system, which consists of a gradient system for the weights of the approximate value function, and a forward Kolmogorov equation. This ODE-PDE system gives an adaptive control for the traffic density. 

The particular forward Kolmogorov equation we obtain is similar to the Boltzmann-Vlasov equation, which has inspired previous traffic models \cite{prigogine1960boltzmann}.

In the remainder of the paper, we prove weak solutions of the ADP ODE-PDE system exist and are unique.

\textbf{Method of proof.} We prove our main result by constructing a sequence of approximating ODE-PDE systems where the PDE involved is linear. We initialize the sequence of solutions by fixing the solution to the PDE at the initial condition as a `zeroth' stage. Then, we take this `approximate' solution and show the ODEs involved have a unique global classical solution. We take the solution to the ODE system, and apply it as input to the linear approximating PDEs. We show that as a result, the approximating parabolic operator we construct are generators of $\omega$-contractive semigroups. This process of fixing solutions and solving equations is iterated. To show that the iteration converges strongly in a suitable sense, the classical Aubin-Lions-Simon \cite{simon1986compact} compactness theorem is invoked. We show the limit obtained solves the ADP ODE-PDE system we constructed uniquely in a weak sense.

We also present some numerical results.

Briefly, the numerical method we use is a first order finite volume method, similar to the one we used in \cite{tirumalai2021robust}. We use this method as it is simple and, importantly, it preserves non-negativity of the density. There are numerous other methods which one can use to simulate evolution systems such as the one we have posed, including the finite element method, and finite difference method. The finite element method is particularly attractive as it is closely tied to the Galerkin method, which can be used to prove existence and uniqueness of solutions to evolution PDEs. So, one could use the same approximating procedure which one uses in theory to compute numerical solutions in practice.

In our final sections, we interpret the results, and in our conclusion describe possible extensions to this problem.

\textbf{Novelty and Contributions.} We make some extensions to the most closely related work in \cite{chevalier2015micro},  \cite{huang2019stabilizing}, and our own work in \cite{tirumalai2021robust}. As opposed to \cite{chevalier2015micro}, we consider second-order dynamics, state constraints, as well as robustness. Both \cite{chevalier2015micro} and \cite{huang2019stabilizing} use a finite-horizon formulation, which we also do in our previous work in \cite{tirumalai2021robust}. 

\cite{huang2019stabilizing} gives a similar mean-field game on first-order dynamics to \cite{chevalier2015micro} for pure AV traffic on a ring of road. Both our previous work in \cite{tirumalai2021robust} and the work in \cite{chevalier2015micro} use a forward-backward iterative method to compute numerical solutions to the dynamic programming equations, but \cite{huang2019stabilizing} uses a numerical optimization approach constructed using Newton's method. These are all offline methods.

As opposed to our own previous work, this paper replaces the finite horizon problem with the infinite horizon problem, and employs approximate dynamic programming to obtain a control \textbf{online} as the density evolves.
\newline
\textbf{Summary.} What we do is as follows:
\begin{enumerate}
	\item Formulate dynamics and robust differential game for finitely many agents;
	\item Formulate the corresponding robust mean-field game;
	\item Obtain DP equations for the robust mean-field game;
	\item Define an approximate value function using weighted basis functions;
	\item Use the DP equations to form dynamics for the weights and traffic density;
	\item Show these dynamics give unique weights and controlled traffic densities;
	\item Demonstrate our ADP using a numerical example, and interpret results.
\end{enumerate}
\noindent
\textbf{Notation} We define the quotient space $\mathbb T:= \mathbb R / \mathfrak L \mathbb Z$, where $\mathfrak L > 0$ is a real constant.  So, the equivalence on $\mathbb T$ is: 
$$
\theta \equiv_{\mathbb T} \theta + \mathfrak L m , \text{ with } m \text{ an arbitrary integer}. 
$$
$\mathcal B(Y)$ is the Borel $\sigma$-algebra associated to the Polish space $Y$. For arbitrary set $\Sigma \in \mathcal B(Y)$, $$\mathbb I_\Sigma(\cdot):Y \rightarrow \{0,1\}$$ is the indicator function of $\Sigma$. $\{\Omega, \mathcal B(\Omega), \mathbb P \}$ is an arbitrary probability space. Spaces $$\xLn{p}(\Omega;Y) := \xLn{p}(\Omega, \mathcal B(\Omega), \mathbb P;Y, \mathcal B(Y), \mu)$$ is a shorthand used for Lebesgue spaces of $p$-integrable random variables which take values in $Y$ under the usual topologies on those spaces. For a random variable $Z:\Omega \rightarrow Y$, we define the $\sigma$-algebra generated by $Z$ as:
$$
\sigma(Z):= \{ \Sigma \in \mathcal B(\Omega) : \Sigma = Z^{-1}(\mathfrak B), \mathfrak B \in \mathcal B(Y) \}.
$$ 
For a filtration $\mathcal F(t) \subset \mathcal B(\Omega)$, $\xLtwo_{\mathcal F}(0,T;\xLtwo(\Omega; Y))$ is the Hilbert space of random processes taking values for each realization in $Y$ which are also adapted to the filtration $\mathcal F(\cdot)$ at each time $t \in \mathbb [0,T]$. $\xLn{4}_{\mathcal F}(\Omega;\xCzero([0,T];Y))$ is the space of pathwise-continuous quartically-integrable $Y$-valued random processes adapted to $\mathcal F(\cdot)$.
$\text{PWC}(\mathbb R_0^+;\mathbb R)$ is the set of piecewise-continuous functions over the non-negative reals taking values over the reals with the supremum norm topology. $\xHn{k}(Y;\mathbb R)$ denotes the Sobolev space \cite{evans2010partial, gilbarg2015elliptic, brezis2011functional} of functions with $k$ weak derivatives in $\xLtwo(Y;\mathbb R)$ under the usual topology. For some normed spaces $U,V$, $\xCzero(U;V)$ are continuous functions with the supremum norm topology:
$
||f||_{\xCzero(U;V)} := \sup_{u \in U} ||f(u)||_V.
$
$\xCzero_b(U;V)$ are the uniformly bounded functions in $\xCzero(U;V)$. The space 
\[
\text{Bounded}_{\mathcal F}(\Omega \times [0,T];\mathbb R)
\] 
are the surely-bounded random processes adapted to $\mathcal F(\cdot)$. This space has norm: $||X||_{\text{Bounded}_{\mathcal F}(\Omega \times [0,T];\mathbb R)}:= \sup_{(\omega, t) \in \Omega \times [0,T]} |X_t(\omega)|$. $\text{Lip}(\mathbb R;\mathbb R)$ are the real-valued Lipschitz functions over the reals under its topology as a H\"older space \cite{gilbarg2015elliptic, evans2010partial}. $\mathcal D(\mathbb R;\mathbb R)$ are the test functions on $\mathbb R$, and $\mathcal D^*(\mathbb R;\mathbb R)$ is the dual space of the test functions.

In this paper, let $Y:= \mathbb T \times [0,s_{max}]$. For a product space $U \times V$,
$$
\partial (U \times V) = (U \times \partial V) \cup (\partial U \times V).
$$
Since $\partial \mathbb T = \emptyset$, $\partial Y = \mathbb T \times \{0, s_{max}\}$. In physical terms, $s_{max}$ is the speed limit for the vehicles.
\section{Problem Formulation}
In this section, we first formulate a finite-size robust constrained-state discounted infinite-horizon stochastic differential game, and then formally (as opposed to rigorously) pass to the mean-field analogue.
\subsection{Finite-size Game}
Take a sequence of i.i.d. random variables $\{(x_0^{i,N}, v_0^{i,N})\}_{i=1}^N$, where $(x_0^{i,N},v_0^{i,N}): \Omega \rightarrow Y$.
Consider a sequence of agents' position-velocity pairs: $\{(x^{i,N}_t, v^{i,N}_t )\}_{i=1}^N \subset 
Y$, with: 
$$(x_{(\cdot)}^{i,N}, v_{(\cdot)}^{i,N}):\Omega \times [0,T]\rightarrow Y.$$ 
Suppose that these follow the It\^o SDEs with reflection:

\begin{equation}
	\begin{split}
		\label{sde1}
		dx_t^{i,N} &= v_t^{i,N} dt; \\ 
		dv_t^{i,N} &= (u^{i,N}_t + w^{i,N}_t) dt  + \nu(v_t^{i,N})dl^{i,N}_t + \sqrt{2\epsilon} \text{ } dW_t^{i,N};\\
		&(x^{i,N}_t,v^{i,N}_t)\Big |_{t = 0}(\omega) = (x_0^{i,N}, v_0^{i,N})(\omega) \in Y \text{ } \forall \text{ } \omega \in \Omega.
	\end{split} 
\end{equation}
\noindent Here, $\epsilon > 0$ is the noise strength. $W^{i,N}_{(\cdot)}$ is a standard scalar Wiener process \cite{oksendal2010stochastic}, and we have $N$ independent copies of these. We assume that these are independent of the initial conditions $\{(x_0^{i,N},v_0^{i,N}) \}_{i=1}^N$. Define: $$\mathcal F^{i,N}(t):= \sigma((x_0^{i,N},v_0^{i,N},W^{i,N}_s)  | s \leq t).$$ $\nu(0) = 1, \nu(s_{max}) = -1, \nu(v) = 0$ otherwise on $(0,s_{max})$, i.e. it is the \textit{inward} pointing direction. \[l_t^{i,N} \in \xLn{4}_{\mathcal F^{i,N}}(\Omega;\xCzero([0,T]; \mathbb R))\] is a surely non-decreasing process (see \cite{tanaka1979stochastic, lions1984stochastic, pilipenko2014introduction} for details) with $l_{t=0}^{i,N} = 0$, and for each $t \in [0,T]$:
$$
\int_0^t \mathbb I_{(0,s_{max})}(v_\tau^{i,N}) dl_\tau^{i,N} = 0, \int_0^t | \nu(v_\tau^{i,N}) | dl^{i,N}_\tau < \infty.
$$
So, the process $l_{(\cdot)}^{i,N}$ increases only when $v_{(\cdot)}^{i,N}$ is on the boundary of $[0,s_{max}]$, i.e. it increases so that the speed limit and direction of travel are enforced on the vehicle dynamics. Define the admissible control and disturbance sets as:
$$\mathbf U:= [-u_{max},u_{max}], \mathbf W:=[-w_{max},w_{max}].$$
$u^{i,N} \in \mathcal U^{i,N} :=\text{Bounded}_{\mathcal F_{i,N}}(\Omega \times [0,T];\mathbf U)$ is the surely bounded acceleration or braking specified by the vehicle's controller, and $w^{i,N} \in \mathcal W^{i,N} :=
\text{Bounded}_{\mathcal F^{i,N}}(\Omega \times [0,T];\mathbf W)$ is an external disturbance, which is also surely bounded. Suppose $0 < w_{max} < u_{max} < \infty$.

If we assume (abusing notation slightly) that we have feedback control and state-dependent disturbances: $$u^{i,N}_t = u^{i,N}(t,x^{i,N}_t,v^{i,N}_t), w^{i,N}_t = w^{i,N}(t,x^{i,N}_t,v^{i,N}_t)$$ and that 
$(u^{i,N},w^{i,N}) \in \text{PWC}([0,T];\text{Lip}(Y;\mathbf U \times \mathbf W))$, there is a unique $\xLn{4}_{\mathcal F^{i,N}}(\Omega;\xCzero([0,T];Y))$ solution to the given SDEs called $(x^{i,N}_{(\cdot)}, v^{i,N}_{(\cdot)})$ \cite{lions1984stochastic, pilipenko2014introduction}.

We define the empirical measure of $\{(x^{j,N}_t, v^{j,N}_t )\}_{j=1\neq i}^N$ by
$
\mu_{(\cdot)}^{i,N} : \Omega \times [0,T] \rightarrow \mathcal P(Y):
$ 
$$
\mu_t^{i,N}(\Sigma):= \frac{1}{N-1}\sum_{j=1\neq i}^N \delta_{(x_t^j,v_t^j)}(\Sigma) \text{ for }
\Sigma \in \mathcal B(Y),
$$
where we have suppressed the dependence on $\omega \in \Omega$, and where $\delta_{(\cdot)}$ is the standard Dirac measure on $Y$. Let $\mathbf x_t^N = (x_t^1, ..., x_t^N)^\top, \mathbf v_t^N = (v_t^1, ..., 
v_t^N)^\top$, and denote the vectors which exclude the $i$-th entries of these vectors by $(\mathbf x_t^{-i,N},\mathbf v_t^{-i,N})$. Similarly, let $\mathbf u(t):=(u^1(t),...,u^N(t))$ and $\mathbf w(t):= (w^1(t),...,w^N(t))$, and let the exlusion of the $i-$th entries be $\mathbf u^{-i,N}, \mathbf w^{-i, N}$. Note that actually:
\[
\mu_t^{i,N} = \mu_t^{i,N}[\mathbf u^{-i,N}, \mathbf w^{-i, N}],
\]
i.e. the empirical measure of the exogenous agents is a functional of their controls and disturbances.
Suppose that the `optimal' control taken by the `$-i$' agents exists and is $\hat{\mathbf{u}}^{-i,N}$, and that the `worst-case' disturbance for the `$-i$' agents exists and is $\hat{\mathbf{w}}^{-i,N}$.

For each agent, we define the following optimal control problems:
\begin{equation}
	\begin{split}
		\label{game_N}
		\inf_{u^{i,N} \in \mathcal U^{i,N}} \sup_{w^{i,N} \in \mathcal W^{i,N}}&\text{ }   \mathbb E \Big [ \mathcal J^{i,N}[u^{i,N},w^{i,N}, \hat{\mathbf{u}}^{-i,N}, \hat{\mathbf{w}}^{-i,N};(x^{i,N}_{t=0}, v^{i,N}_{t=0}), t=0] \Big ]. \\ 
		&\text{s.t.} \text{ } (\ref{sde1})
	\end{split}
\end{equation}
So, the agent attempts to make the best decision while accounting for the worst disturbance it can face, assuming that the other agents already make the best decisions they can subject to the worst disturbances they  face (see \cite{bacsar1998dynamic, basar2018handbook} for similar formulation descriptions).

Here, the cost functional is:

\begin{equation*}
	\begin{split}
		&\mathcal J^{i,N}[u^{i,N},w^{i,N}, \hat{\mathbf{u}}^{-i,N}, \hat{\mathbf{w}}^{-i,N};(x^{i,N}_{\tau}, v^{i,N}_{\tau}), \tau]:= ... \\
		& \int_{\tau}^\infty e^{-\alpha s} \mathcal  L^{i,N}(x_s^{i,N}, v^{i,N}_s, u^{i,N}_s, w^{i,N}_s,\mu_s^{i,N}[\hat{\mathbf{u}}^{-i,N}, \hat{\mathbf{w}}^{-i,N}]) \text{ }ds,
	\end{split}
\end{equation*}

where

\begin{equation*}
	\begin{split}
		&\mathcal L^{i,N}(x^{i,N},v^{i,N}, u^{i,N},w^{i,N},\mu^{i,N}):= ... \\ 
		& \frac{1}{2}(u^{i,N})^2 -  \frac{1}{2\gamma^2} (w^{i,N})^2 + (\int_Y \phi(x^{i,N},\eta_1)d\mu^{i,N}(\eta_1, \eta_2) - \frac{1}{\beta})v^{i,N}.
	\end{split}
\end{equation*}

\noindent
In this cost, the first term accounts for ride comfort, i.e. the vehicle is penalized for accelerating too quickly. The second term rewards the controller for causing the disturber to apply its best opposition. In the third term, the integral term causes the controller to slow down or speed up vehicles where there is high congestion. The multiplication of $\beta^{-1}v^{i,N}$ enforces a preference for the vehicle to move as quickly as possible. So, there are a number of competing objectives the controller attempts to balance.

We could apply dynamic programming to each agent's optimal control problem, but this would lead to a large system of coupled dynamic programming  PDEs, and, assuming that the solution to such a system exists, this would lead to a fully centralized optimal control \cite{basar2018handbook, friedman1972stochastic}. Instead, we use mean field games.
\subsection{Mean-field Game}
Suppose $N\rightarrow \infty$, so we have an infinite sequence of agents' position-velocity pairs: $\{(x_t^i, v_t^i)\}_{i=1}^\infty \subset Y$, with 
$(x_{(\cdot)}^i, v_{(\cdot)}^i) : \Omega \times [0,T] \rightarrow Y$. From this sequence of agents, we exclude one anonymous representative agent $i$, and drop its indexing. This agent is subject to the reflected It\^o SDEs:
\begin{equation}
	\begin{split}
		\label{sde2}
		dx_t &= v_t dt ;  \\ 
		dv_t &= (u_t + w_t) dt + \nu(v_t)dl_t + \sqrt{2\epsilon} \text{ } dW_t ; \\
		&(x_t,v_t)\Big |_{t = 0}(\omega) = (x_0, v_0)(\omega) \in Y \text{ } \forall \text{ } \omega \in \Omega.
	\end{split} 
\end{equation}
Assume that 
\[
u \in \mathcal U :=\text{Bounded}_{\mathcal F}(\Omega \times [0,T];\mathbf U),w \in \mathcal W :=
\text{Bounded}_{\mathcal F}(\Omega \times [0,T];\mathbf W),
\]
\noindent 
where each of the quantities here are anonymized analogues of what appear in (\ref{sde1}). 
Define:
\[
\mathcal F(t):= \sigma((x_0,v_0,W_s) | s \leq t).
\]
Again, assume $x_0,v_0$ are independent of $W_{(\cdot)}$.
Similarly to before, assume (again abusing notation slightly) that we have feedback control and state-dependent disturbances: 
\[
u_t = u(t,x_t,v_t), w_t = w(t,x_t,v_t).
\] 
If $(u,w) \in \text{PWC}([0,T];\text{Lip}(Y;\mathbf U \times \mathbf W))$, we have a $\xLn{4}_{\mathcal F}(\Omega;\xCzero ([0,T]; Y))$ solution $(x_{(\cdot)},v_{(\cdot)})$ \cite{lions1984stochastic}.
Suppose the optimal control taken by all of the exogenous `$-i$' agents exists and is $u^*: \mathbb R^+_0 \times Y \rightarrow \mathbf U$ and the worst-case disturbance they encounter exists and is $w^*: \mathbb R^+_0 \times Y \rightarrow \mathbf W$. We will specify for what problems these are optimal shortly in (\ref{mfg}). These \textbf{anonymized} exogenous agents each follow \textbf{independent} copies of:
\begin{equation}
	\begin{split}
		d\Xi_t &= \Upsilon_t dt; \\ 
		d\Upsilon_t &=  (u^*(t, \Xi_t, \Upsilon_t) + w^*(t,\Xi_t, \Upsilon_t) )dt + \nu(\Upsilon_t) \text{ } dL_t + \sqrt{2\epsilon} \text{ } dB_t; \\
		&(\Xi_t,\Upsilon_t)\Big |_{t = 0}(\omega) = (x_0, v_0)(\omega) \in Y \text{ } \forall \text{ } \omega \in \Omega,
	\end{split} 
\end{equation}
where $B_{(\cdot)}$ is another anonymized scalar Wiener process independent of $W_{(\cdot)}$ and $x_0,v_0$. 
Define:
\[
\mathcal F^*(t):= \sigma((x_0,v_0,B_s) | s \leq t).
\]
If $(u^*, w^*) \in \text{PWC}([0,T];\text{Lip}(Y;\mathbf U \times \mathbf W))$, then  $(\Xi_{(\cdot)}, \Upsilon_{(\cdot)}) \in \xLn{4}_{\mathcal F^*}(\Omega;\xCzero([0,T]; Y))$. For convenience, let $y_{(\cdot)}:= (\Xi_{(\cdot)},\Upsilon_{(\cdot)})$.

We let $m(t)\equiv m[u^*,w^*](t)$ be the distribution of $y_t$ (and therefore a functional of the optimal control and worst-case disturbance), and assume that $\mu_t^{i,N} \rightarrow m(t)$ weakly-* in $\mathcal P(Y)$ for each $i \in \mathbb N_1$, $t \in [0,T]$. Recall that $y_1 \equiv_{\mathbb T} y_1 + \mathcal L$. We can show that the density of $m$ called $\rho:[0,T] \times Y \rightarrow \mathbb R$ satisfies the forward Kolmogorov (FK) equation over $[0,T]\times Y $:
\begin{equation}
	\begin{split}
		\label{measures}
		&\partial_t \rho(t,\cdot)+ \mathfrak A (u^*,w^*)\rho(t,\cdot) =  0 \text{ in } (0,T] \times \text{int } Y ; \\ &\rho(0,\cdot) = \rho_0 \text{ in } \text{int } Y; \\
		& \rho (t,y_1,y_2) = \rho (t,y_1 + \mathfrak L,y_2) \text{ in } [0,T] \times Y; \\
		& \partial_{y_2} \rho (t,y_1,0) = \partial_{y_2} \rho (t,y_1,s_{max}) = 0 \text{ in } [0,T] \times \mathbb T,
	\end{split} 
\end{equation}
where we have assumed that everywhere on $\partial  Y$, $u^*(t,\cdot) = w^*(t,\cdot) = 0$, and:
\begin{equation}
	\begin{split}
		&  \mathfrak A (u,w)f(\cdot) := -\epsilon \partial_{y_2}^2 f(\cdot) + \nabla_{y} \cdot [(y_2, u(\cdot) + w(\cdot))^\top  f(\cdot)].
	\end{split}
\end{equation}
This evolution equation closely resembles a Boltzmann-Vlasov equation of gas dynamics subject to diffusion and some accelerations $u^*, w^*$ \cite{neunzert1984introduction}.
The boundary conditions for this case would usually be of Robin type:
\[
\mathbf n^\top((y_2, u^*(t, y) + w^*(t, y))^\top \rho(t,y) - \begin{bmatrix}
	0 & 0 \\ 0 & \epsilon
\end{bmatrix} \nabla_x f(t,y)) = 0 \text{ on } \partial Y,
\]
but due to the assumption we made (and what we see later in the solution of our optimal control problem) the first term involving the drift drops out, leaving only the normal derivative, which, due to the structure of $\partial Y$, becomes the simple homogeneous Neumann condition we expressed earlier.

Define:
\begin{equation*}
	\begin{split}
		&\Lambda(y_1,y_2,u,w, m):= \frac{1}{2}u^2 - \frac{1}{2\gamma^2}w^2 + (\int_{Y} \phi(y_1,\eta_1) dm(\eta_1,\eta_2) - \beta^{-1})y_2,
	\end{split}
\end{equation*}
which is s.t. $\mathcal L^{i,N}(y_1,y_2, u,w, \mu_t^{i,N}) \rightarrow \Lambda(y_1,y_2,u,w,m(t))$ for fixed $y_1,y_2, u,w$ under the convergence assumption for $\mu_{(\cdot)}^{i,N}$ and $m(\cdot)$ we made earlier in this section. The mean-field game cost accounts for the same features that the one for the finite-size game does.

We formulate the mean-field game:
\begin{equation}
	\begin{split}
		\label{mfg}
		\inf_{u \in \mathcal U} \sup_{w \in \mathcal W}&\text{ }   \mathbb E \Big [ \mathcal J[u,w, {u}^{*}, {w}^{*};(x_{t=0}, v_{t=0}), t=0] \Big ], \\ 
		&\text{s.t.} \text{ } (\ref{sde2}) \\
	\end{split}
\end{equation}
where
\begin{equation}
	\begin{split}
		\mathcal J[u,w, {u}^{*}, {w}^{*};(x_{\tau}, v_{\tau}), \tau]:= \int_{\tau}^\infty e^{-\alpha  s}\Lambda(x_s,v_s, u_s, w_s, m[u^*,w^*](s)) ds,
	\end{split}
\end{equation}
with $\alpha > 0$ as the discount factor.
As noted in the introduction, the constraint that $v_t^* \in [0,s_{max}]$ for $t \in [0,T]$ is automatically satisfied by an optimal state trajectory $(x^*_{(\cdot)}, v^*_{(\cdot)})$ due to the structure of the dynamics (\ref{sde2}).

Note that although this optimal control problem is formulated for a single agent, calling this a game is correct for these reasons:
\begin{enumerate}
	\item This is a min-max problem, which can be interpreted as a game between a player and the environment;
	\item The exogenous agents' dynamics are included through the mean field measure $m(\cdot) \equiv m[u^*,w^*](\cdot)$, which is the probability distribution of the agents who are observing the optimal control and worst environmental disturbance. So, the representative is playing against the mean-field;
	\item The representative is itself anonymous, so the same problem can be posed for any agent indexed by $i \in \mathbb N_1$. We have dropped the indexing since it does not actually matter.
\end{enumerate}
In particular, the assumption that the representative agent is playing against the best decisions taken by the rest of the players (here, the other players are the other vehicles and the environment) is a standard assumption of game theory. See \cite{bacsar1998dynamic, basar2018handbook, bensoussan2013mean} for details. We now give the dynamic programming system which solves the mean-field game if a suitable solution exists.

One should also note that as a part of dynamic programming for the mean-field game, we realize that the distribution of the representative agent $(x_{(\cdot)}, v_{(\cdot)})$ when subject to the optimal control is the same as the mean-field distribution $m(\cdot)$ \cite{bensoussan2013mean}. So, we refer to either one interchangeably.

Define the pre-Hamiltonian:
\begin{equation}
	\mathcal H(y,u, w, m,p):= \Lambda(y, u, w, m) + p^\top g(y, u, w),
\end{equation}
where $g(y, u, w):= (y_2, u + w)^\top$. 

We now write the mean-field game system of PDEs which gives the solution of the previous robust mean-field game. Note that we replace the $\inf \sup$ formulation with a $\min \max$. The respective problems turn out to be rather trivial. Indeed they are scalar problems soluble by simple applications of the extreme value theorem.
\begin{proposition}
	
	The stationary value function and and density  $(\mathcal V, P)$, $P$ the density of $\mathbf m$, corresponding to the solution of the given robust mean-field game satisfy the system of PDEs:
	\begin{equation}
		\begin{split}
			\label{mfg-static}
			&\alpha  \mathcal V + H(y, \mathbf m, \nabla_{y}\mathcal V) + \epsilon \partial_{y_2}^2 \mathcal V = 0 \text{ in } \text{int } Y; \\
			& \alpha (P - \rho_0) + \mathfrak A(u^*, w^*) P  = 0  \text{ in } \text{int } Y ;\\ 
			&H(y,\mathbf m,p):=  \min_{u \in \mathbf U} \max_{w \in \mathbf W } \mathcal H(y, u, w, \mathbf m,p); \\
			& w^*(u) = \arg \max_{w \in \mathbf W} \mathcal H(y, u, w, \mathbf m,\nabla_{y}\mathcal V) ; \\ 
			&u^*  = \arg \min_{u \in \mathbf U} \mathcal H(y, u, w^*(u), \mathbf m,\nabla_{y}\mathcal V); \\
			& \mathcal V (y_1,y_2) = \mathcal V (y_1 + \mathfrak L,y_2) \text{ in } Y; \\
			& \partial_{y_2} \mathcal V (y_1,0) = \partial_{y_2} \mathcal V (y_1,s_{max}) = 0 \text{ in } \mathbb T; \\
			& P (y_1,y_2) = P (y_1 + \mathfrak L,y_2) \text{ in } Y; \\
			& \partial_{y_2} P (y_1,0) = \partial_{y_2} P (y_1,s_{max}) = 0 \text{ in } \mathbb T,
		\end{split}
	\end{equation}
	assuming a solution exists and is regular-enough.
\end{proposition}

\begin{proof}
	(sketch)
	For details, see \cite{bensoussan2013mean}. In short, one assumes that a solution to a finite-horizon version of the given mean-field game can be extended to arbitrarily long times, and $P$ is defined by:
	$$
	P(\cdot) := \alpha \int_0^\infty e^{-\alpha t} \rho(t,\cdot) dt
	$$
	as in Chapter 7 of \cite{bensoussan2013mean}. Then, standard dynamic programming approaches to the discounted infinite-horizon optimal control problem are used. 
	
	As we obtained in \cite{tirumalai2021robust} by solving the optimization problem on the pre-Hamiltonian, the optimal control and worst-case disturbance are the ramp functions with cutoff:
	\begin{equation}
		\label{optimizer1}
		u^*(p_2):= \begin{cases} 
			u_{max}&\text{ if } p_2 < -u_{max} \\
			-p_2 &\text{ if } -u_{max} \leq p_2 \leq  u_{max} \\
			-u_{max} &\text{ if } u_{max} < p_2 
		\end{cases}
	\end{equation}
	\begin{equation}
		\label{optimizer2}
		w^*(p_2):= \begin{cases} 
			-w_{max} &\text{ if } \gamma^2 p_2 < - w_{max} \\
			\gamma^2 p_2 &\text{ if } -w_{max} \leq \gamma^2 p_2 \leq w_{max} \\
			w_{max} &\text{ if } w_{max} < \gamma^2 p_2.
		\end{cases}
	\end{equation}
	These constrained solutions are found by checking if the unconstrained solution is in the interior of the feasible sets, and if not, one then compares the values at the boundary points. The solutions are unique as the respective problems are concave and convex.
\end{proof}

Note that in this system, the initial condition $\rho_0$ enters into the equation for the density. So, different controls are required for different initial traffic distributions. Conventionally, this system of PDEs is computed offline. Depending on the desired fine detail on the control and access to computing hardware on a specific autonomous vehicle, controls obtained offline might not be useful as the road conditions change, and might not be updated quickly enough. However, the robustness accounts for some of these issues. 

So, we employ an approximate dynamic programming approach that can be computed online as the traffic density evolves.

\section{Approximate Dynamic Programming}
In this section, we introduce a method of approximating the solution to the mean-field game system (\ref{mfg-static}). Let the approximate value function $\tilde{\mathcal V}^{A,B}: [0,T] \times Y  \rightarrow \mathbb R$ be the Fourier series:
\begin{equation}
	\begin{split}
		& \tilde{\mathcal V}^{A,B}(t,y) :=   \sum_{i,j=0}^{K-1} \big ( a_{ij}(t) \sin(\frac{2 \pi i y_1}{\mathfrak L }) + b_{ij}(t)\cos( \frac{2 \pi i y_1}{\mathfrak L}) \big) \cos(\frac{2\pi j  y_2}{s_{max}} ).
	\end{split}
\end{equation}

This choice for an approximate value function satisfies the periodic and Neumann boundary conditions we desire for the value function, and hence also the state constraints. This is the main reason for choosing this approximate value function. In addition, it is linear in the weights, which makes computing gradients w.r.t. the weights much simpler.

For problems where state constraints do not enter, one could pick sigmoid functions or other common basis or activation functions used in machine learning. However, for state constrained problems, these might lead to violation of boundary conditions for the value function. One could penalize violation of boundary conditions instead, but since basis functions which exactly satisfy the boundary conditions (state constraints) were so readily available, we simply picked them.

$A,B: [0,T] \rightarrow \mathbb R^{K \times K}$ will be subject to dynamics which we will specify shortly. 

Unfortunately, due to the control and disturbance constraints, the optimal control (\ref{optimizer1}) and worst-case disturbance (\ref{optimizer2}) are not differentiable w.r.t. the co-state $p$. This leads to issues with uniqueness when we form gradient-based update rules for the approximate value function weights. So, for the optimization problem on the pre-Hamiltonian, we select the following feasible, smooth sub-solutions:
\begin{equation}
	\tilde u(p_2):= -u_{max} \text{tanh}(p_2)
\end{equation}
\begin{equation}
	\tilde w(p_2):= w_{max} \text{tanh}(\gamma^2 p_2).
\end{equation}
These are sigmoidal approximations to the optimal cutoff ramp function solutions.

Now, we define the approximate Hamilton-Jacobi-Bellman-Isaacs (HJB-I) cost (or residual error):
\begin{equation}
	\begin{split}
		E (A, B,\rho) :=  ||\alpha \tilde{\mathcal V}^{A,B} + \tilde H(\cdot, m, \nabla_{y} 
		\tilde{\mathcal V}^{A,B}) + \epsilon \partial_{y_2}^2 \tilde{\mathcal V}^{A,B}||_{\xLtwo (Y;\mathbb R)}^2,
	\end{split}
\end{equation}
where the approximate Hamiltonian ${\tilde{H}}$ is:
\begin{equation}
	\tilde H (y, m, p)  := \mathcal H(y, \tilde u(p_2), \tilde w(p_2), m, p).
\end{equation}
Using these quantities, we define the following ADP ODE-PDE system (suppressing some arguments):
\begin{equation}
	\begin{split}
		\label{learning}
		&\xDrv  {A(t)}{t} = -\theta^{-1}\nabla_{A} E (A(t), B(t),\rho(t,\cdot)), \text{ }A(0) = A_{0}; \\
		&\xDrv{B(t)}{t} = -\theta^{-1}\nabla_{B} E(A(t), B(t),\rho(t,\cdot)), \text{ }B(0) = B_{0}; \\
		&\partial_t \rho(t,\cdot) = - \mathcal A(t;\tilde u(\partial_{y_2} \tilde{\mathcal V}^{A,B}), 
		\tilde w(\partial_{y_2}  \tilde{\mathcal V}^{A,B})) \rho(t,\cdot) \text{ in }  \text{int }Y \times (0, T]; \\
		& \rho(0,\cdot) = \rho_0 \text{ in } Y; \\
		& \rho(t,y_1,y_2) = \rho(t,y_1 + \mathfrak L, y_2) \text{ in } [0,T] \times Y; \\
		& \partial_{y_2} \rho(t, y_1, 0) = \partial_{y_2} \rho(t, y_1, s_{max}) = 0 \text{ in } [0,T] \times \mathbb T, \\
	\end{split} 
\end{equation}
where $0 < \theta^{-1} << 1$ is the learning rate, and:
\begin{equation}
	\begin{split}
		\label{operator}
		&  -\mathcal A (t;u,w)f(t,y) := \epsilon \partial_{y_2}^2 f(t,y) - \nabla_{y} \cdot [(y_2, u(t,y) + w(t,y))^\top  f(t,y)].
	\end{split}
\end{equation}
We define the domain and range of linear operator $-\mathcal A(t;u,w)$ shortly. Note that the PDE and ODEs together form a semi-linear degenerate parabolic system.
\section{Theoretical Results}
Define the Hilbert space for integer $p \geq 0$: 
\begin{equation}
	\mathbb H^p:=  \mathbb R^{K \times K} \times \mathbb R^{K \times K} \times \xHn{p}(Y;\mathbb R),
\end{equation}
with norm ($||\cdot||_F$ is the Frobenius norm):
\begin{equation}
	||(A, B,\rho)||_{\mathbb H^p}^2:=  ||A||_{F}^2 + 
	||B||_{F}^2 + ||\rho||_{\xHn{p}(Y;\mathbb R)}^2,
\end{equation}
which are induced by the usual inner products in the respective spaces.
$\xHn{0}(Y;\mathbb R)$ is just $\xLtwo(Y;\mathbb R)$. Let:
\begin{equation}
	\mathbf L := \begin{bmatrix} 0 & 0 \\ 0 & \epsilon \end{bmatrix}.
\end{equation}
Define: 
\begin{equation}
	\begin{split}
		D(\mathcal A)&:=  \{ f \in \xHn{2}(Y;\mathbb R) : (\nabla_y^\top f) \mathbf L \mathbf n |_{\partial Y} = 0 \text{ almost everywhere}\},
	\end{split}
\end{equation}
which is the (vector) space of $\xHn{2}(Y;\mathbb R)$ functions satisfying the Neumann BCs in the trace sense.
Recall the earlier definition in (\ref{operator}) of the form of linear time-varying operator $-\mathcal A(t;u,w)$ and consider it to be of type:
\begin{equation}
	-\mathcal A(t;u,w):  D(\mathcal A) \rightarrow \xLtwo(Y;\mathbb R),
\end{equation} 
with $u,w$ assumed to satisfy the control-disturbance constraints, and act as parameters for the (unbounded) time-varying linear operator $-\mathcal A(\cdot)$. Also, define $\mathfrak G, \mathfrak H: \mathbb H^{2} \rightarrow \mathbb R^{K \times K}$ as
\begin{equation}
	\mathfrak G(A, B, \rho):= -\theta^{-1}\nabla_{A} E (A, B, \rho), 
\end{equation}
\begin{equation}
	\mathfrak H(A, B, \rho):= -\theta^{-1}\nabla_{B} E (A, B, \rho).
\end{equation}
\begin{definition}
	A weak solution to (\ref{learning}) is a triplet $(A,B,\rho) \in \xCzero([0,T];\mathbb H^1)$ which is s.t. 
	$\forall \text{ } \psi \in \mathcal D(Y;\mathbb R)$:
	\begin{equation}
		\label{def_1}
		A(t) = A_0 + \int_0^t \mathfrak G(A(s), B(s), \rho(s)) \text{ } ds;
	\end{equation}
	\begin{equation}
		\label{def_2}
		B(t) = B_0 + \int_0^t \mathfrak H(A(s), B(s), \rho(s)) \text{ } ds;
	\end{equation}
	\begin{equation}
		\begin{split}
			\label{def_3}
			&( \rho(t), \psi )_{\mathcal D^*(Y;\mathbb R), \mathcal D(Y;\mathbb R)}= ( \rho_0, \psi )_{\mathcal D^*(Y;\mathbb R), \mathcal D(Y;\mathbb R)} + ... \\ 
			& \int_0^t (\rho(s), -\mathcal A^*(s;\tilde u(\partial_{y_2} \tilde{\mathcal V}),\tilde w(\partial_{y_2} \tilde{\mathcal V}) ) \psi )_{\mathcal D^*(Y;\mathbb R), \mathcal D(Y;\mathbb R)} \text{ } ds;
		\end{split} 
	\end{equation}
	\begin{equation}
		\begin{split}
			\label{def_4}
			(\nabla_y^\top \rho) \mathbf L \mathbf n = 0 \text{ almost everywhere on } [0,T] \times \partial Y.
		\end{split} 
	\end{equation}
	Here,
	\begin{equation*}
		\begin{split}
			&( \rho ,-\mathcal A^*(t;u,w)\psi )_{\mathcal D^*(Y;\mathbb R),\mathcal D(Y;\mathbb R)} := \int_Y \rho \epsilon  \partial_{y_2}^2 \psi  + \rho  y_2 \partial_{y_1} \psi + \rho(u + w) \partial_{y_2} \psi \text{ } dy.
		\end{split} 
	\end{equation*}
\end{definition}
We now state precisely the main theoretical result of this paper.
\begin{theorem}
	If $\rho_0 \in \xHtwo(Y;\mathbb R)$ is s.t. $(\nabla^\top_y \rho_0) \mathbf L \mathbf n = 0$ a.e. on $\partial Y$, $A_0, B_0 \in \mathbb R^{K \times K}$, then the ADP ODE-PDE system (\ref{learning}) has a unique weak solution: 
	\begin{equation}
		(A, B, \rho)(\cdot) \in \xCzero([0,T];\mathbb H^{1}) \text{ for arbitrary but finite $T > 0$.}
	\end{equation}
\end{theorem}
\begin{proof}
	\textbf{Outline of proof.} First, we establish some preliminary results on Lipschitzianity of a number of relevant quantities, and that a sequence of approximating operators to $-\mathcal A(\cdot)$ generate $\omega$-contractive semigroups. Then, we construct sequences of approximate solutions to the ADP ODE-PDE system where we first fix the solution of the FK equation and show the weight equations have a classical solution, and then fix the weights and show the FK equation has a classical solution (in a suitable Sobolev space). This is done via semigroup methods. This procedure results in bounds for the solutions, which we use along with the Aubin-Lions-Simon compactness theorem to conclude that a limit point exists. Finally, we show this limit solves the ADP ODE-PDE system weakly, which is done by taking advantage of the preliminary results on Lipschitzianity. Some of the preliminary results include rather elementary steps, but we include any steps we consider substantive for completeness.
	
	\textbf{Step 0. Preliminary Results.}
	\begin{lemma}
		For every $h_1:= (A_1,B_1, \rho_1),h_2:=(A_2,B_2,\rho_2) \in \mathfrak B \subset \mathbb H^2$, $\mathcal \mathfrak$ an arbitrary bounded set in $\mathbb H^2$, there is a constant $L$ s.t.
		
		\begin{equation}
			||\mathfrak G(h_1) - \mathfrak G (h_2)||_F \leq \theta^{-1} L ||h_1 - h_2||_{\mathbb H^2};
		\end{equation}
		\begin{equation}
			||\mathfrak H(h_1) - \mathfrak H (h_2)||_F \leq  \theta^{-1} L ||h_1 - h_2||_{\mathbb H^2}.
		\end{equation}
	\end{lemma}
	\begin{proof}
		We will focus on $\mathfrak G$, and note that the same results follow analogously for $\mathfrak H$, as the structures are virtually the same. 
		First, recall that for functions defined on Banach spaces $U,V,W$, $f : U \rightarrow V, g : V \rightarrow W$, if $f,g$ are Fr\'echet differentiable, then their composition $g(f(\cdot)): U \rightarrow W$ is also Fr\'echet differentiable, and its Fr\'echet derivative is:
		\begin{equation}
			D_u[g(f(u))] h = D_v[g(f(u))] D_u[f(u)] h
		\end{equation}
		for arbitrary direction $h \in U$ at a particular point $u \in U$. This is the chain rule for Fr\'echet derivatives \cite{cheney2001analysis}. This is simply composition of the Fr\'echet derivatives for the different functions. We can show that:
		\begin{equation}
			D_f [||f||_{\xLtwo(Y;\mathbb R)}^2]h = 2 \langle f, h \rangle_{\xLtwo(Y;\mathbb R)},
		\end{equation}
		and letting:
		\begin{equation}
			\begin{split}
				&f(A,B,\rho):= \alpha \tilde{\mathcal V}^{A,B} + \mathcal H(\cdot, \tilde u(\partial_{y_2}\tilde{\mathcal V}^{A,B}), \tilde w(\partial_{y_2}\tilde{\mathcal V}^{A,B}),m, \nabla_{y} 
				\tilde{\mathcal V}^{A,B}) + \epsilon \partial_{y_2}^2 \tilde{\mathcal V}^{A,B},
			\end{split}
		\end{equation}
		we can inspect $\mathcal H$ and conclude that each term in $f$ is actually smooth in $A$.
		$\tilde{\mathcal V}^{A,B}$ and 
		$\partial_{y_2}^2 \tilde{\mathcal V}^{A,B}$ are linear in $A$, and the terms of $\mathcal H$ which depend on $A$ are compositions of $\partial_{y_2}\tilde{\mathcal V}^{A,B}$ with smooth functions, and are hence smooth in $A$.
		\newline
		So, we differentiate normally and obtain:
		
		\begin{equation*}
			\begin{split}
				&\partial_{a_{ij}}f(A,B,\rho) = ... \\ 
				&\alpha \partial_{a_{ij}}\tilde{\mathcal V}^{A,B} +  \Big ( \partial_u \mathcal H(\cdot, \tilde u(\partial_{y_2} \tilde{\mathcal V}^{A,B}), 
				\tilde w(\partial_{y_2}  \tilde{\mathcal V}^{A,B}), m, \nabla_{y} 
				\tilde{\mathcal V}^{A,B})\partial_{p_2}\tilde u(\partial_{y_2} \tilde{\mathcal V}^{A,B}) + ... \\
				&\partial_w \mathcal H(\cdot, \tilde u(\partial_{y_2} \tilde{\mathcal V}^{A,B}), 
				\tilde w(\partial_{y_2}  \tilde{\mathcal V}^{A,B}), m, \nabla_{y} 
				\tilde{\mathcal V}^{A,B})\partial_{p_2}\tilde w(\partial_{y_2} \tilde{\mathcal V}^{A,B}) \Big ) \partial_{a_{ij},y_2}^2 \tilde{\mathcal V}^{A,B} + ... \\
				&\nabla^\top_{p} \mathcal H(\cdot, \tilde u(\partial_{y_2} \tilde{\mathcal V}^{A,B}), 
				\tilde w(\partial_{y_2}  \tilde{\mathcal V}^{A,B}), m, \nabla_{y} 
				\tilde{\mathcal V}^{A,B})\partial_{a_{ij}} \nabla_{y} 
				\tilde{\mathcal V}^{A,B} + \epsilon \partial_{a_{ij}}\partial_{y_2}^2 \tilde{\mathcal V}^{A,B},
			\end{split}
		\end{equation*}
		
		which are the elements $(\nabla_A f(A,B,\rho))_{ij} = \partial_{a_{ij}}f(A,B,\rho)$. In the above, 
		
		\begin{equation*}
			\partial_{a_{ij}}\tilde{\mathcal V}^{A,B} = \sin(\frac{2 \pi y_1 i}{\mathfrak L}) \cos(\frac{2 \pi y_2 j}{s_{max}});
		\end{equation*}

		\begin{equation*}
			\partial_{y_2}\tilde{\mathcal V}^{A,B} = -\sum_{i=0}^{K-1} \sum_{j = 0}^{K-1} \Big (a_{ij}\sin(\frac{2 \pi y_1 i}{\mathfrak L}) + b_{ij} \cos(\frac{2 \pi y_1 i}{\mathfrak L}) \Big ) \sin(\frac{2 \pi y_2 j}{s_{max}}) \frac{2 \pi  j}{s_{max}};
		\end{equation*}

		\begin{equation*}
			\partial^2_{a_{ij},y_2} \tilde{\mathcal V}^{A,B} = -\sin(\frac{2 \pi y_1 i}{\mathfrak L}) \sin(\frac{2 \pi y_2 j}{s_{max}}) \frac{2 \pi j}{s_{max}};
		\end{equation*}

		\begin{equation*}
			\partial_u \mathcal H(\cdots,\nabla_y \tilde{\mathcal V}^{A,B}) = \tilde u(\partial_{y_2} \tilde{\mathcal V}^{A,B}) + \partial_{y_2} \tilde{\mathcal V}^{A,B};
		\end{equation*}

		\begin{equation*}
			\partial_w \mathcal H(\cdots,\nabla_y \tilde{\mathcal V}^{A,B}) = -\frac{1}{\gamma^2}\tilde w(\partial_{y_2} \tilde{\mathcal V}^{A,B}) + \partial_{y_2} \tilde{\mathcal V}^{A,B};
		\end{equation*}

		\begin{equation*}
			\partial_{p_2} \tilde u(\partial_{y_2} \tilde{\mathcal V}^{A,B}) = -u_{max} \text{sech}^2(\partial_{y_2} \tilde{\mathcal V}^{A,B});
		\end{equation*}

		\begin{equation*}
			\partial_{p_2} \tilde w(\partial_{y_2} \tilde{\mathcal V}^{A,B}) = \gamma^2 w_{max} \text{sech}^2(\gamma ^2\partial_{y_2} \tilde{\mathcal V}^{A,B});
		\end{equation*}

		\begin{equation*}
			\partial_{a_{ij}}\nabla_y^\top \tilde{\mathcal V}^{A,B} = \Big(\cos(\frac{2 \pi y_1 i}{\mathfrak L})\cos(\frac{2 \pi y_2 j }{s_{max}})\frac{2 \pi i}{\mathfrak L}, -\sin(\frac{2 \pi y_1 i}{\mathfrak L})\sin(\frac{2 \pi y_2 j }{s_{max}})\frac{2 \pi j}{s_{max}}\Big );
		\end{equation*}

		\begin{equation*}
			\partial_{y_2}^2 \tilde{\mathcal V}^{A,B} =  -\sum_{i=0}^{K-1} \sum_{j = 0}^{K-1} \Big (a_{ij}\sin(\frac{2 \pi y_1 i}{\mathfrak L}) + b_{ij} \cos(\frac{2 \pi y_1 i}{\mathfrak L}) \Big ) \cos(\frac{2 \pi y_2 j}{s_{max}}) (\frac{2 \pi  j}{s_{max}})^2;
		\end{equation*}

		\begin{equation*}
			\nabla_p^\top \mathcal H(\cdots) = (y_2, \tilde u(\partial_{y_2} \tilde{\mathcal V}^{A,B}) +\tilde w(\partial_{y_2} \tilde{\mathcal V}^{A,B}));
		\end{equation*}

		\begin{equation*}
			\partial_{a_{ij}}\partial_{y_2}^2 \tilde{\mathcal V}^{A,B} = - \sin(\frac{2 \pi y_1 i}{\mathfrak L}) \cos(\frac{2 \pi y_2 j}{s_{max}})(\frac{2 \pi j}{s_{max}})^2.
		\end{equation*}
		
		For $h \in \mathbb R^{K \times K}$, the Fr\'echet (directional) derivative of $f(\cdot,B, \rho)$ in $A$ is:
		
		\begin{equation*}
			D_A [f(A,B,\rho)]h = \langle \nabla_A f(A,B,\rho), h \rangle_{F}.
		\end{equation*}
		
		Now, using the chain rule, we obtain by composing the Fr\'echet derivatives:
		\begin{equation}
			\begin{split}
				D_A[E(A,B,\rho)]h &= 2\langle f(A,B,\rho), \langle \nabla_A f(A,B,\rho), h \rangle_F \rangle_{\xLtwo (Y;\mathbb R)} \\
				&= \langle 2\int_Y f(A,B,\rho) \nabla_A f(A,B,\rho) \text{ }dy, h \rangle_F \\
				&=: \langle \nabla_A E(A,B,\rho), h \rangle_F
			\end{split}
		\end{equation}
		for arbitrary direction $h \in \mathbb R^{K \times K}$ at point $A \in \mathbb R^{K \times K}$. Let 
		
		\begin{equation*}
			M:= \max_{p_2 \in \mathbb R} | p_2 \text{sech}^2(p_2)|.
		\end{equation*}
		\noindent
		Using a simple bounding procedure, we know $0< M < 1$. Using the derivatives we computed for the terms of $\partial_{a_{ij}}f(A,B,\rho)$, we can show by crudely bounding term-by-term, component-by-component:
		\begin{equation}
			\begin{split}
				&|D_A[f(A,B,\rho)]h | \leq
				K^2\Big [ \alpha + [(u_{max} + M)u_{max} + (w_{max} + M)w_{max}]\frac{2 \pi K}{s_{max}}  + ... \\
				& (s_{max} \frac{2 \pi K}{\mathfrak L} + (u_{max} + w_{max})\frac{2 \pi K}{s_{max}}) + \epsilon (\frac{2 \pi K}{s_{max}})^2 \Big ] ||h||_F =: C_1 ||h||_F.
			\end{split}
		\end{equation}
		
		\noindent
		Also, we can perform a similar procedure on $f$ to obtain:
		\begin{equation}
			\begin{split}
				\label{bounds}
				&||f||_{\xCzero_b(Y;\mathbb R)} \leq ... \\ 
				&(\alpha  + \epsilon (\frac{2 \pi K}{s_{max}})^2 + \frac{2 \pi K}{\mathfrak L}s_{max} + ... \\ 
				&\frac{2 \pi K}{s_{max}}(u_{max} + w_{max}))(||A||_F + ||B||_F)K^2 + ... \\
				&\frac{1}{2}u_{max}^2 + \frac{1}{2\gamma^2}w_{max}^2 + (||\phi||_{\xCzero_b(Y;\mathbb R)} + \beta^{-1})s_{max} =: C_2(A,B).
			\end{split}
		\end{equation}
		
		To show that $\mathfrak G$ is Lipschitz, we show that $D_A [D_A E(A,B,\rho)], D_B [D_A E(A,B,\rho)]$, and  $D_\rho [D_A E(A,B,\rho)]$ are bounded bilinear operators. Before we compute these derivatives, we give some more results for Fr\'echet derivatives. For a bilinear functional $\mathfrak F:(U \times V=: W) \rightarrow Z$, $Z$ a Banach space, the joint Fr\'echet derivative of $\mathfrak F$ is:
		\[
		D_w [\mathfrak F(u,v)] h = \mathfrak F(h_1,v) + \mathfrak F(u, h_2)
		\]
		for direction $h = (h_1,h_2) \in U \times V = W$. If $u: X \rightarrow U, v:X \rightarrow V$, $X$ a Banach space, then from our earlier chain rule:
		\[
		D_x [\mathfrak F(u(x),v(x))]h = \mathfrak F(D_x [u(x)]h, v(x)) + \mathfrak F(u(x), D_x [v(x)]h)
		\]
		for direction $h \in X$.
		The bilinear functional we are interested in is $$\langle \cdot, \cdot \rangle_{\xLtwo(Y;\mathbb R)} : \xLtwo(Y;\mathbb R) \times \xLtwo(Y;\mathbb R) \rightarrow \mathbb R.$$ Let $X:= \mathbb H^2$, $U = V := \xLtwo(Y;\mathbb R)$, $W:= \mathbb R$, $x = (x_1,x_2,x_3) = (A, B,\rho)$. Note that for $x = (x_1,...,x_n) \in X_1 \times ... \times X_n$:
		\begin{equation}
			D_x [f(x)] h = \sum_{i=1}^n D_{x_i}f(x)h_i
		\end{equation}
		with $D_{x_i}f(x)h_i$ being the Fr\'echet derivatives on the individual spaces $X_i$. These derivatives are termed partial Fr\'echet derivatives. So, if a function has each partial Fr\'echet derivative defined, then the total Fr\'echet derivative is defined. Applying the differentiation rule for bilinear functionals and the chain rule,
		\begin{equation}
			\begin{split}
				D_A[D_A[E(A, B,\rho)h_1 ] h_2] &= \langle D_A[f(A, B,\rho)]h_1, \langle \nabla_A f(A, B,\rho), h_2 \rangle_F \rangle_{\xLtwo(Y;\mathbb R)} + ... \\
				&\langle f(A, B,\rho), D_A [\langle \nabla_A f(A, B,\rho), h_1 \rangle_F] h_2 \rangle_{\xLtwo(Y;\mathbb R)}.
			\end{split}
		\end{equation}
		The form of the first term is already known from our previous computations. We can also show using our previous computations that:
		
		\begin{equation*}
			\begin{split}
				&\partial_{a_{kl}}\partial_{a_{ij}}f(A, B,\rho)=\partial_{a_{ij}}\partial_{a_{kl}}f(A, B,\rho) = ... \\ 
				&\Big ( -[\partial_{p_2}\tilde u(\partial_{y_2} \tilde{\mathcal V}^{A,B})\partial_{a_{kl}}\partial_{y_2}\tilde{\mathcal V}^{A,B} + \partial_{a_{kl}}\partial_{y_2}\tilde{\mathcal V}^{A,B} ]\partial_{p_2}\tilde u(\partial_{y_2} \tilde{\mathcal V}^{A,B} ) - ... \\ 
				& [\tilde u(\partial_{y_2} \tilde{\mathcal V}^{A,B}) + \partial_{y_2} \tilde{\mathcal V}^{A,B}]\partial_{p_2}^2 \tilde u(\partial_{y_2} \tilde{\mathcal V}^{A,B})\partial_{a_{kj}}\partial_{y_2} \tilde{\mathcal V}^{A,B} + ... \\
				&[-\frac{1}{\gamma^2}\partial_{p_2}\tilde w(\partial_{y_2} \tilde{\mathcal V}^{A,B})\partial_{a_{kl}}\partial_{y_2} \tilde{\mathcal V}^{A,B} + \partial_{a_{kl}}\partial_{y_2} \tilde{\mathcal V}^{A,B}]\partial_{p_2} \tilde w(\partial_{y_2} \tilde{\mathcal V}^{A,B}) + ... \\ 
				&[-\frac{1}{\gamma^2}\tilde w(\partial_{y_2} \tilde{\mathcal V}^{A,B}) + \partial_{y_2} \tilde{\mathcal V}^{A,B}]\partial_{p_2}^2 \tilde w(\partial_{y_2} \tilde{\mathcal V}^{A,B})\partial_{a_{kj}}\partial_{y_2} \tilde{\mathcal V}^{A,B} \Big)\partial_{a_{ij}}\partial_{y_2}\tilde{\mathcal V}^{A,B} + ... \\
				&(\partial_{p_2}\tilde u(\partial_{y_2}\tilde{\mathcal V}^{A,B}) + \partial_{p_2}\tilde w(\partial_{y_2}\tilde{\mathcal V}^{A,B}))\partial_{a_{kl}}\partial_{y_2}\tilde u(\partial_{y_2}\tilde{\mathcal V}^{A,B})\partial_{a_{ij}}\partial_{y_2}\tilde u(\partial_{y_2}\tilde{\mathcal V}^{A,B}).
			\end{split}
		\end{equation*}
		
		This function can be uniformly bounded as:
		
		\begin{equation*}
			\begin{split}
				&\partial_{a_{kl}}\partial_{a_{ij}}f(A, B,\rho) \leq \Big [(u_{max}\frac{2\pi K}{s_{max}} + \frac{2\pi K}{s_{max}} + u_{max}\frac{4}{3 \sqrt{3}} + \tilde M)u_{max}\frac{2\pi K}{s_{max}} + ... \\ 
				& (w_{max}\frac{2\pi K}{s_{max}} + \frac{2\pi K}{s_{max}} + u_{max}\frac{4}{3\sqrt{3}}+\tilde M)\gamma^2 w_{max}\frac{2\pi K}{s_{max}} + (u_{max} + w_{max}\gamma^2)(\frac{2\pi K}{s_{max}})^2\Big] \\
				&=: C_3,
			\end{split}
		\end{equation*}
		
		where $\tilde M:=\max_{p_2}|2 p_2  \text{sech}^2(p_2)\text{tanh}(p_2)|$. This is obtained by determining upper bounds term-by-term.
		Let 
		\[
		[(\nabla_A \otimes \nabla_A)f(A, B,\rho)]_{ijkl} = \partial_{a_{ij}}\partial_{a_{kl}}f(A, B,\rho),
		\] 
		which is the Hessian tensor, which for fixed $(A, B,\rho)$ is $\mathbb R^{K \times K \times K \times K}$-valued. For $\mathfrak M \in \mathbb R^{K \times K \times K \times K}$ and $\mathbf V \in \mathbb R^{K \times K}$, let:
		\[
		\mathfrak M \mathbf V = \sum_{i,j,k,l=0}^{K-1}  \mathfrak m_{ijkl} \mathbf v_{kl} \mathbf e_{ij}
		\]
		where $\mathbf e_{ij} \in \mathbb R^{K \times K}$ are the standard basis for $\mathbb R^{K \times K}$, i.e. matrices with $1$ in the $ij$-th entry. Now, entrywise:
		\[
		D_{a_{kl}}[D_{a_{ij}} E(A, B,\rho)h_{1,ij}]h_{2,kl} = \partial_{a_{kl}}\partial_{a_{ij}}f(A, B,\rho)h_{1,ij}h_{2,kl},
		\]
		which we sum up to conclude:
		\[
		D_A [\langle \nabla_A f(A, B,\rho), h_{1} \rangle_F]h_2 = \langle (\nabla_A \otimes \nabla_A)f(A, B,\rho)h_1, h_2 \rangle_{F}.
		\]
		From the form of $\partial_{a_{kl}}\partial_{a_{ij}}f(\rho,A,B)$ have:
		\[
		|D_{A}[D_{A} f(A, B,\rho)h_1]h_2| \leq K^4 C_3 ||h_1||_F ||h_2||_F
		\]
		thus:
		\[
		|D_A[D_A[E(A, B,\rho)h_1]h_2]| \leq \mathfrak L s_{max}(C_1^2 + C_2(A,B) C_3 K^4 ) ||h_1||_F ||h_2||_F.
		\]
		Now, since $\mathfrak B$ is bounded, define $\mathfrak R:= \sup_{b \in \mathfrak B \subset \mathbb H^2}||b||_{\mathbb H^2}$. Then, $\mathfrak B \subset B(0, \mathfrak R)$, an open ball centered on the origin. So, for every $(A, B,\rho) \in \mathfrak B$,
		\[
		C_2(A,B) \leq C_2(\mathfrak R \mathbf I, \mathfrak R \mathbf I) =: C_4. 
		\]
		Thus, in any bounded set $\mathfrak B \subset \mathbb H^2$, $\mathfrak G$ is boundedly differentiable in $A$.
		We cam employ a similar procedure on $D_B[D_A[E(A, B,\rho)]h_1]h_2$. Similarly to what we computed earlier:
		
		\begin{equation*}
			\begin{split}
				&\partial_{b_{kl}}\partial_{a_{ij}}f(A, B,\rho) = \Big ( -[\partial_{p_2}\tilde u(\partial_{y_2} \tilde{\mathcal V}^{A,B})\partial_{b_{kl}}\partial_{y_2}\tilde{\mathcal V}^{A,B} + \partial_{b_{kl}}\partial_{y_2}\tilde{\mathcal V}^{A,B} ]\partial_{p_2}\tilde u(\partial_{y_2} \tilde{\mathcal V}^{A,B} ) - ... \\ 
				& [\tilde u(\partial_{y_2} \tilde{\mathcal V}^{A,B}) + \partial_{y_2} \tilde{\mathcal V}^{A,B}]\partial_{p_2}^2 \tilde u(\partial_{y_2} \tilde{\mathcal V}^{A,B})\partial_{b_{kl}}\partial_{y_2} \tilde{\mathcal V}^{A,B} + ... \\
				&[-\frac{1}{\gamma^2}\partial_{p_2}\tilde w(\partial_{y_2} \tilde{\mathcal V}^{A,B})\partial_{b_{kl}}\partial_{y_2} \tilde{\mathcal V}^{A,B} + \partial_{b_{kl}}\partial_{y_2} \tilde{\mathcal V}^{A,B}]\partial_{p_2} \tilde w(\partial_{y_2} \tilde{\mathcal V}^{A,B}) + ... \\ 
				&[-\frac{1}{\gamma^2}\tilde w(\partial_{y_2} \tilde{\mathcal V}^{A,B}) + \partial_{y_2} \tilde{\mathcal V}^{A,B}]\partial_{p_2}^2 \tilde w(\partial_{y_2} \tilde{\mathcal V}^{A,B})\partial_{b_{kl}}\partial_{y_2} \tilde{\mathcal V}^{A,B} \Big)\partial_{a_{ij}}\partial_{y_2}\tilde{\mathcal V}^{A,B} + ... \\
				&(\partial_{p_2}\tilde u(\partial_{y_2}\tilde{\mathcal V}^{A,B}) + \partial_{p_2}\tilde w(\partial_{y_2}\tilde{\mathcal V}^{A,B}))\partial_{b_{kl}}\partial_{y_2}\tilde u(\partial_{y_2}\tilde{\mathcal V}^{A,B})\partial_{a_{ij}}\partial_{y_2}\tilde u(\partial_{y_2}\tilde{\mathcal V}^{A,B}),
			\end{split}
		\end{equation*}
		\noindent
		and indeed as before, bounding this term-by-term yields:
		\[
		|\partial_{b_{kl}}\partial_{a_{ij}}f(A, B,\rho)| \leq C_3.
		\]
		Let 
		\[
		[(\nabla_B \otimes \nabla_A)f(A, B,\rho)]_{ijkl} = \partial_{b_{kl}}\partial_{a_{ij}}f(A, B,\rho),
		\]
		and similarly to before:
		
		\begin{equation}
			\begin{split}
				D_B[D_A[E(A, B,\rho)h_1]h_2] &= \langle \langle \nabla_B f(A, B,\rho), h_1 \rangle_F, \langle \nabla_A f(A, B,\rho), h_2 \rangle_F \rangle_{\xLtwo(Y;\mathbb R)} + ... \\
				&\langle f(A, B,\rho), \langle (\nabla_B \otimes \nabla_A)f(A, B,\rho) h_1,h_2 \rangle_F \rangle_{\xLtwo(Y;\mathbb R)},
			\end{split}
		\end{equation}
		\noindent 
		and:
		\[
		|D_B[D_A[E(A, B,\rho)h_1]h_2] | \leq \mathfrak L s_{max}(C_1^2 + C_2(A,B) C_3 K^4 ) ||h_1||_F ||h_2||_F.
		\]
		So, $\mathfrak G$ is boundedly differentiable in $\mathfrak B$. Finally, we turn to the differential in $\rho$.
		\[
		D_\rho[D_A f(A, B,\rho)h]h_\rho = \langle D_\rho f(A, B,\rho)h_\rho, D_A f(A, B,\rho) h\rangle_{\xLtwo(Y;\mathbb R)} 
		\]
		for arbitrary directions $h \in \mathbb R^{K \times K}, h_\rho \in \xLtwo(Y;\mathbb R)$. The second term present in the other gradient drops out as $\partial_{a_{ij}}f(\rho,A,B)$ is constant in $\rho$, and:
		\[
		D_\rho [f(A, B,\rho)]h_{\rho} = \int_Y \phi(\cdot,\eta_1)h_\rho(\eta_1,\eta_2) d\eta.
		\]
		We can bound this with:
		\[
		|D_\rho [f(A, B,\rho)]h_{\rho}| \leq ||\phi||_{\xCzero_b(Y;\mathbb R)} |\langle \mathbb I_Y(\cdot), h_\rho \rangle_{\xLtwo(Y;\mathbb R)}| \leq \mathfrak L s_{max} ||\phi||_{\xCzero_b(Y;\mathbb R)} ||h_\rho||_{\xLtwo(Y;\mathbb R)},
		\]
		so:
		\[
		|D_\rho[D_A f(A, B,\rho)h]h_\rho| \leq \mathfrak L s_{max} ||\phi||_{\xCzero_b(Y;\mathbb R)} C_1 ||h||_F ||h_\rho||_{\xLtwo(Y;\mathbb R)}.
		\]
		Combining the bounds for these three derivatives, it follows:
		\[
		|D_x [D_A [E(A, B,\rho)]h_1]h_2| \leq (2C_5 + \mathfrak L s_{max} ||\phi||_{\xCzero_b(Y;\mathbb R)} C_1)||h_1||_F ||h_2||_{\mathbb H^2}
		\]
		where $C_5:= \mathfrak L s_{max}(C_1^2 + C_4 C_3 K^4 ) $. Now, consider any two points $h_1:=(A_1,B_1,\rho_1),h_2:=(A_2,B_2,\rho_2)$ with $h_1,h_2 \in B$. Take convex combinations:
		\[
		\eta(t):= t h_1 + (1-t)h_2
		\]
		with $t \in [0,1]$. From the mean-value theorem,
		\[
		|D_A [E(h_1) - E(h_2)]h|=|\langle \nabla_A E(h_1) - \nabla_A E(h_2),h\rangle_F| \leq L(h) ||h_1 - h_2||_{\mathbb H^2},
		\]
		where 
		\[
		L(h):=\sup_{t \in [0,1]} ||D_x [D_A E(\eta(t))h]||_{\mathcal L(\mathbb H^2;\mathbb H^2) } \leq (2C_5 + ||\phi||_{\xCzero_b(Y;\mathbb R)} C_1)||h||_F),
		\] 
		which implies that:
		\[
		||\nabla_A E(h_1) - \nabla_A E(h_2)||_{F^*}  \leq 
		L ||h_1 - h_2||_{\mathbb H^2}
		\]
		where 
		\[
		L:=\sup_{h \in B(0,1) \subset \mathbb R^{K \times K}}L(h) = (2C_5 + ||\phi||_{\xCzero_b(Y;\mathbb R)} C_1),
		\]
		and  $||\cdot||_{F^*}$ is the dual norm of the Frobenius norm, defined as:
		\[
		||\mathbf X||_{F^*}:= \sup_{\mathbf Y \in B(0,1) \subset \mathbb R^{K \times K}} |\langle \mathbf X, \mathbf Y  \rangle_F|.
		\]
		Indeed, the space $\{\mathbb R^{K\times K}, ||\cdot||_F\}$ is self-dual, so:
		\[
		||\nabla_A E(h_1) - \nabla_A E(h_2)||_{F}  \leq 
		L_1 ||h_1 - h_2||_{\mathbb H^2}
		\]
		\noindent
		which proves the claim for $\mathfrak G$. Virtually the same argument can be applied to $\mathfrak H$, replacing $\partial_{a_{ij}}$ with $\partial_{b_{ij}}$ as needed. This concludes the proof of Lemma 1.
	\end{proof}
	\begin{lemma}
		Let $\mathcal L_n(t;u,w) : D(\mathcal L_n) \subset \xLtwo(Y;\mathbb R) \rightarrow \xLtwo(Y;\mathbb R)$ be:
		\begin{equation}
			\label{approximator}
			\mathcal L_n(t;u,w)(\cdot):=- \mathcal A(t;u,w)(\cdot) + \frac{1}{n}\nabla_y^2 (\cdot),
		\end{equation}
		\begin{equation}
			\begin{split}
				D(\mathcal L_n)&:=  \{ f \in \xHtwo(Y;\mathbb R) : (\nabla_y^\top f) \mathbf L_n \mathbf n |_{\partial Y} = 0 \text{ almost everywhere}\}.
			\end{split}
		\end{equation}
		where 
		\[
		\mathbf L_n:= \begin{bmatrix} \frac{1}{n} & 0 \\ 
			0 & \frac{1}{n} + \epsilon
		\end{bmatrix}.
		\]
		If 
		\[
		(u,w) \in \xCone([0,T] \times Y;\mathbf U \times \mathbf W)), 
		\] 
		then for each $f \in D(\mathcal L_n),$ $n \geq 1$:
		
		\[
		\mathcal L_n(t;u,w)f \in \xCone([0,T];\xLtwo(Y;\mathbb R)).
		\]
	\end{lemma}
	\begin{proof} 		
		Take $\{f_m\}_{m=1}^\infty \subset \xCinfty(Y;\mathbb R)$ s.t. $f_m \rightarrow f$ in $D(\mathcal L_n)$ under $||\cdot||_{\xHtwo(Y;\mathbb R)}$. Then,
		\[
		[\mathcal L_n(t+h;u,w) - \mathcal L_n(t;u,w)]f_m = -[\mathcal A(t+h;u,w) - \mathcal A(t;u,w)]f_m
		\]
		\[
		= -\partial_{y_2}[(u(t+h,\cdot) - u(t,\cdot) + w(t+h,\cdot) - w(t,\cdot))f_m] .
		\]
		So, using smoothness several times,
		\[
		\lim_{h\rightarrow 0} \frac{1}{h} [\mathcal L_n(t+h;u,w) - \mathcal L_n(t;u,w)]f_m
		\]
		\[
		= 		-\lim_{h\rightarrow 0} \frac{1}{h} \partial_{y_2}[(u(t+h,\cdot) - u(t,\cdot) + w(t+h,\cdot) - w(t,\cdot))f_m]
		\]
		$$
		= -\partial_{y_2}[(\partial_t u(t,\cdot) + \partial_t w(t,\cdot))f_m ] 
		\rightarrow -\partial_{y_2}[(\partial_t u(t,\cdot) + \partial_t w(t,\cdot))f ] 
		$$
		in $\xLtwo(Y;\mathbb R)$ from the earlier convergence of $f_m \rightarrow f$ in $D(\mathcal L_n)$ under $||\cdot||_{\xHtwo(Y;\mathbb R)}$, which completes the proof.
	\end{proof}
	\begin{remark}
		\label{remark1}
		Note that from the structure of $\mathbf L_n$ and of $\partial Y$, $D(\mathcal A) = D(\mathcal L_n)$.
	\end{remark}
	\begin{lemma}
		The family of operators $\{\mathcal L_n(t;u,w)\}_{t\in [0,T]}$ defined in (\ref{approximator}) with domain $D(\mathcal L_n)$ is a stable family of generators of $\omega-$contractive semigroups on $\xLtwo(Y;\mathbb R)$  for some $\omega > 0$ at every $n \in \mathbb N_1$ for every $$(u,w) \in \xCone([0,T] \times {\text{cl }Y};\mathcal U \times \mathcal W)$$ s.t. $u(t,\cdot) = w(t,\cdot) = 0 \text{ on } \partial Y$ for every $t \in [0,T]$.
	\end{lemma}
	
	\begin{proof}
		The proof largely follows the proof of Theorem 7.4.5. of \cite{evans2010partial}, modified to treat the Neumann conditions we have in our problem.
		
		\noindent
		\textbf{Step 1. Density of $D(\mathcal L_n)$ in $\xLtwo(Y;\mathbb R)$.}
		
		Recall our Remark \ref{remark1}. Note that  $\mathcal D(Y;\mathbb R) \equiv \xCinfty_0(Y;\mathbb R)$ is dense in $L^2(Y;\mathbb R)$, and that $\xCinfty_0(Y;\mathbb R) \subset D(\mathcal L_n) \subset \xLtwo(Y;\mathbb R)$. So, $D(\mathcal L_n)$ is dense in $\xLtwo(Y;\mathbb R)$.
		
		\noindent
		\newline
		\newline
		\textbf{Step 2. Closure of $\mathcal L_n(t;u,w)$.} 	
		
		The operator $\mathcal L_n(t;u,w)$ is closed iff its dense domain $D(\mathcal L_n)$ is a Banach space under:
		\[
		||f||_{D(\mathcal L_n)}:= ||f||_{\xLtwo(Y;\mathbb R)} + ||\mathcal L_n(t;u,w)f||_{\xLtwo(Y;\mathbb R)},
		\]
		which is termed the graph norm.
		
		Clearly, $D(\mathcal L_n)$ is a linear space. $||\cdot||_{\xLtwo(Y;\mathbb R)}$ is obviously a norm over $D(\mathcal L_n)$, as $D(\mathcal L_n) \subset \xLtwo(Y;\mathbb R)$. If $||\mathcal L_n(t;u,w)(\cdot)||_{\xLtwo(Y;\mathbb R)}$ is a seminorm over $D(\mathcal L_n)$, then the graph norm shall be a true norm, as the inclusion of $||\cdot||_{\xLtwo(Y;\mathbb R)}$ will take care of the positive-definite/point-separating property we require. From the triangle inequality for the $\xLtwo(Y;\mathbb R)$ norm,
		\[
		||\mathcal L_n(t;u,w)[f_1 + f_2]||_{\xLtwo(Y;\mathbb R)} \leq ||\mathcal L_n(t;u,w)f_1||_{\xLtwo(Y;\mathbb R)} + 
		||\mathcal L_n(t;u,w)f_2||_{\xLtwo(Y;\mathbb R)},
		\]
		and from the linearity of $\mathcal L_n(t;u,w)f$, and absolute homogeneity of $\xLtwo(Y;\mathbb R)$:
		\[
		||\mathcal L_n(t;u,w)[cf]||_{\xLtwo(Y;\mathbb R)} = |c| ||\mathcal L_n(t;u,w)f||_{\xLtwo(Y;\mathbb R)}.
		\]
		Thus, $||\mathcal L_n(t;u,w)(\cdot)||_{\xLtwo(Y;\mathbb R)}$ is a seminorm, and the graph norm is a norm over $D(\mathcal L_n)$.
		So, $\{D(\mathcal L_n), ||\cdot||_{D(\mathcal L_n)} \}$ is a normed space. 
		
		Now we show $D(\mathcal L_n)$ is complete under the topology of $||\cdot||_{D(\mathcal L_n)}$. The following computations are somewhat routine, but we include them for the purpose of rigor and completeness. 
		
		Take a sequence $\{f_n \}_{n=1}^\infty \subset D(\mathcal L_n)$ s.t. for some $l,m > M$:
		\[
		||f_l - f_m||_{D(\mathcal L_n)} < \tilde \epsilon
		\]
		i.e. $\{f_m \}_{m=1}^\infty$ is Cauchy in the space $\{D(\mathcal L_n), ||\cdot||_{D(\mathcal L_n)} \}$. Then, we can take some $0 < \epsilon_1, \epsilon_2 < \tilde \epsilon$ s.t. $\epsilon_1 + \epsilon_2 = \tilde \epsilon$, so:
		\[
		||f_l- f_m||_{D(\mathcal L_n)} = ||f_l - f_m||_{\xLtwo(Y;\mathbb R)} + ||\mathcal L_n(t;u,w)[f_l - f_m]||_{\xLtwo(Y;\mathbb R)} < \epsilon_1 + \epsilon_2,
		\]
		so let:
		\[
		||f_l - f_m||_{\xLtwo(Y;\mathbb R)} < \epsilon_1, ||\mathcal L_n(t;u,w)[f_l - f_m]||_{\xLtwo(Y;\mathbb R)} < \epsilon_2.
		\]
		Then, there is a re-ordering $\{f_{m_j}\}_{j=1}^\infty$ s.t.
		\[
		||f_{m_{j+1}} - f_{m_{j}}||_{\xLtwo(Y;\mathbb R)} < 2^{-j}, ||\mathcal L_n(t;u,w)[f_{m_{j+1}} - f_{m_j}]||_{\xLtwo(Y;\mathbb R)} < 2^{-j}.
		\]
		Define:
		\[
		g_{q}:= |f_{m_1}| + \sum_{j=1}^q |f_{m_{j+1}} - f_{m_{j}}|.
		\]
		Then,
		\[
		||g_{q}||_{\xLtwo(Y;\mathbb R)} \leq ||f_{m_{1}}||_{\xLtwo(Y;\mathbb R)} + \sum_{j=1}^q ||f_{m_{j+1}} - f_{m_{j}}||_{\xLtwo(Y;\mathbb R)}.
		\]
		\[
		\leq ||f_{m_{1}}||_{\xLtwo(Y;\mathbb R)} + \sum_{j=1}^\infty ||f_{m_{j+1}} - f_{m_{j}}||_{\xLtwo(Y;\mathbb R)}
		\]
		\[
		= ||f_{m_{1}}||_{\xLtwo(Y;\mathbb R)} + 1 =: C.
		\]
		$g_q$ is thus measurable, and $0 \leq g_1 \leq g_2 \leq ...$ and so on. Thus, $0 \leq g_1^2 \leq g_2^2 \leq ...$, and so on.
		So, $||g_q||_{\xLtwo(Y;\mathbb R)}^2 \leq C^2$. Let  $g:= \lim_{q \rightarrow \infty} g_q$. $g$ is also measurable, and by the monotone convergence theorem,
		\[
		||g||_{\xLtwo(Y;\mathbb R)}^2 = \lim_{q \rightarrow \infty}||g_q||^2_{\xLtwo(Y;\mathbb R)} \leq C^2,
		\]
		and $g(x) < \infty$ almost everywhere. So, $f_{m_1}(x) + \sum_{j=1}^l f_{m_{j+1}}(x) - f_{m_{j}}(x)$ converges absolutely almost everywhere, and $\lim_{j \rightarrow \infty} f_{m_j}(x)$ exists almost everywhere. Let:
		\[
		f(x):= \begin{cases}
			\lim_{j \rightarrow \infty} f_{m_j}(x) &\text{ if } g(x) < \infty \\
			0 &\text{ else }
		\end{cases}.
		\]
		So, $f$ is measurable, and $f_{m_{j}}(x) \rightarrow f(x)$ almost everywhere, and $|f| \leq g$. Thus,
		$||f||_{\xLtwo(Y;\mathbb R)} \leq ||g||_{\xLtwo(Y;\mathbb R)} \leq C^2$, so $f \in \xLtwo(Y;\mathbb R)$. Now,
		\[
		|f_{m_j} - f| \leq (|f_{m_j}| + |f|)^2 \leq 4 |g|^2
		\]
		and $|f_{m_j}(x) - f(x)| \rightarrow 0$ almost everywhere, so by the dominated convergence theorem,
		\[
		||f_{m_j} - f||_{\xLtwo(Y;\mathbb R)}^2 \rightarrow 0.
		\]
		Now, to show that if $||\mathcal L_n(t;u,w)[f_m - f_l]||_{\xLtwo(Y;\mathbb R)} < \epsilon_2$, then
		\[
		||\mathcal L_n(t;u,w)[f_m - f]||_{\xLtwo(Y;\mathbb R)} \rightarrow 0,
		\]
		define:
		\[
		\phi_m := 	\mathcal L_n(t;u,w)f_m,
		\]
		and
		\[
		\gamma_q := |\phi_{m_1}| + \sum_{j=1}^q |\phi_{m_{j+1}} - \phi_{m_{j}}|.
		\]
		Then, in the computations we performed earlier involving $f_{m_j}$, $g_q$, replace $f_{m_j}$ by $\phi_{m_j}$, $g_q$ by $\gamma_q$, and so on, and the conclusion of the previous argument follows for $||\mathcal L_n (\cdots)[\cdot]||_{\xLtwo(Y;\mathbb R)}$.
		
		Now, since $\mathcal L_n(t;u,w)f \in \xLtwo(Y;\mathbb R)$, from the construction of the operator $\mathcal L_n(t;u,w)$, we also have that $\partial^2_{y_i} f$ are in $\xLtwo(Y;\mathbb R)$. $Y$ is already circular in $y_1$, and in the other, we can consider periodic extensions about $y_2 = 0$ and $y_2 = s_{max}$. So, for each function $\psi \in \xLtwo(Y;\mathbb R)$, we can represent $\psi$ as a Fourier series (now using $i$ as the imaginary unit):
		\[
		\psi(\cdot) = \sum_{l,m \in \mathbb Z} c_{l,m} e^{(\lambda_{l,1} (\cdot)_1 + \lambda_{m,2}(\cdot)_2 )i }
		\]
		where $\lambda_{l,1}:= 2 \pi l/\mathcal L, \lambda_{m,2} := 2 \pi m/s_{max}$, and:
		\[
		c_{l,m}:= \frac{1}{\mathfrak L s_{max}}\langle \psi,  e^{(\lambda_{l,1}(\cdot)_1 + \lambda_{m,2}(\cdot)_2 )i }\rangle_{\xLtwo(Y;\mathbb R)}.
		\]
		The Fourier series of course only converges in the sense of $L^2(Y;\mathbb R)$ to $\psi$ rather than other stronger senses.
		From Parseval-Plancharel, we have for $\psi \in L^2(Y;\mathbb R)$:
		\[
		||\psi||_{\xLtwo(Y;\mathbb R)}^2 = \sum_{l,m \in \mathbb Z} |c_{l,m}|^2.
		\]
		If $\partial_{y_1}^{\alpha} \psi \in \xLtwo(Y;\mathbb R)$, then:
		\[
		\partial_{y_1}^{\alpha}\psi = \sum_{l,m \in \mathbb Z} c_{l,m} (i \lambda_{l,1})^{\alpha}e^{(\lambda_{l,1} y_1 + \lambda_{m,2}y_2 )i }
		\]
		and analogously for $\partial_{y_2}^\alpha$. Also:
		\[
		||\partial_{y_1}^{\alpha} \psi||_{\xLtwo(Y;\mathbb R) }^2 =\sum_{l,m \in \mathbb Z} |c_{l,m}|^2| \lambda_{l,1}|^{2\alpha},
		\]
		and again analogously for $\partial_{y_2}$.
		
		Note that for functions $\psi \in \xHn{p}(Y;\mathbb R)$, $p \in \mathbb N_1$, using periodic extensions, we have that:
		\[
		\sum_{l,m \in \mathbb Z} |c_{l,m}|^2 (1 + |\lambda_l^{m,\top} \lambda_l^m|^{p})   < \infty
		\]
		where $\lambda_l^m = (\lambda_{l,1}, \lambda_{m,2})^\top$  \cite{taylor2010partial}.

		So, let $f = \psi$. We know $f, \partial_{y_1}^2 f, \partial_{y_2}^2 f \in \xLtwo(Y;\mathbb R)$. It follows:
		\[
		\sum_{l,m \in \mathbb Z} |c_{l,m}|^2(1 + \lambda_{l,1}^{4} + \lambda_{m,2}^{4}) < \infty.
		\]
		Now, since $2 \lambda_{l,1}^{2}\lambda_{m,2}^{2} < 1 +  \lambda_{l,1}^{4} + \lambda_{m,2}^{4}$:
		\[
		\sum_{l,m \in \mathbb Z} |c_{l,m}|^2 2(\lambda_{l,1}^{2}\lambda_{m,2}^{2}) < \infty,
		\]
		Thus, 
		\[
		\sum_{l,m \in \mathbb Z} |c_{l,m}|^2 (1 + |\lambda_l^{m,\top} \lambda_l^m|^{2})  < \infty.
		\]
		So, $f \in \xHtwo(Y;\mathbb R)$. Now, to check that $f$ satisfies the boundary conditions, let:
		\[
		b_m:= (\nabla_y^\top f_m) \mathbf L_n \mathbf n |_{\partial Y}, b:= (\nabla_y^\top f)\mathbf L_n \mathbf n |_{\partial Y}
		\] 
		Since $f_m \in D(\mathcal L_n)$, $b_m = 0$ almost everywhere. Thus,
		\[
		-e^{-m} < ||b_m||_{\xLtwo(\partial Y;\mathbb R)}^2 < e^{-m}.
		\]
		By the squeeze theorem, it follows that:
		\[
		\lim_{m \rightarrow \infty} ||b_m||_{\xLtwo(\partial Y;\mathbb R)}^2 = 0,
		\]
		and by the dominated convergence theorem (using $g(y) = 1$ to dominate $b_m^2$ on $\partial Y$), it follows that:
		\[
		\lim_{m \rightarrow \infty} ||b_m||_{\xLtwo(\partial Y;\mathbb R)}^2 = ||b||_{\xLtwo(\partial Y;\mathbb R)}^2,
		\]
		so
		\[
		||b||_{\xLtwo(\partial Y;\mathbb R)} = 0,
		\]
		hence $b = (\nabla_y^\top f)\mathbf L_n \mathbf n |_{\partial Y} = 0$ almost everywhere, thus $f \in D(\mathcal L_n)$, and the normed linear space $$\{D(\mathcal L_n), ||\cdot||_{D(\mathcal L_n)} \}$$ is a Banach space, so $\mathcal L_n(t;u,w)$ is closed.
		
		\noindent
		\textbf{Step 3. Invertibility of the resolvent $\mathcal R(\lambda;\mathcal L_n(t;u,w))$.}
		
		The resolvent operator is:
		\[
		\mathcal R (\lambda;\mathcal L_n(t;u,w))f := (\lambda \mathcal I - \mathcal L_n(t;u,w))^{-1}f \text{ for } f \in \xLtwo(Y;\mathbb R),
		\]
		for $\lambda \in \rho(\mathcal L_n(t;u,w))$, the resolvent set of $\mathcal L_n(t;u,w)$, which is the set of $\lambda \in \mathbb R$ s.t. $$\lambda \mathcal I - \mathcal L_n(t;u,w): D(\mathcal L_n) \rightarrow \xLtwo(Y;\mathbb R)$$ is invertible, i.e. every $\lambda \in \mathbb R$ where the elliptic problem:
		\begin{equation}
			\begin{split}
				\label{elliptic}
				&(\lambda \mathcal I - \mathcal L_n(t;u,w)) \phi = f \text{ in } \text{int }Y; \\
				&(\nabla_y^\top \phi) \mathbf L_n \mathbf n = 0 \text{ on } \partial Y
			\end{split}
		\end{equation}
		has a unique solution $\phi \in D(\mathcal L_n)$ for any given $f \in \xLtwo(Y;\mathbb R)$. It is a simple matter to verify that $\lambda \mathcal I - \mathcal L_n(t;u,w)$ is strongly uniformly elliptic.  Select (based on the conditions in \cite{ern2021finite}): 
		\[
		\lambda \geq \omega > |\text{ess} \inf_{y \in Y} \partial_{y_2} (u(t,y)+w(t,y))| + \frac{1}{2}(s_{max} + \text{ess} \sup_{y \in Y} [u(t,y)+w(t,y)]).
		\]
		\noindent
		From the conditions on $u,w$ and from periodicity, we have that for $\phi \in \xCinfty (\text{cl } Y;\mathbb R)$:
		
		\begin{equation}
			\begin{split}
				\label{div_theorem}
				\int_Y \phi(y) [(y_2,u(t,y)+w(t,y) \nabla_y \phi(y))] \text{ } dy = - \frac{1}{2} \int_Y \nabla_y \cdot (y_2, u(t,y) + w(t,y)) \phi^2 (y) \text{ } dy,
			\end{split}
		\end{equation}
		\noindent
		which also holds for $\phi \in \xHone(Y;\mathbb R)$ by density. Define the bilinear form 
		\[
		B[\cdot,\cdot;\lambda]:\xHone(Y;\mathbb R) \times \xHone(Y;\mathbb R) \rightarrow \mathbb R:
		\]
		
		\begin{equation}
			\begin{split}
				B[\phi,\psi;\lambda] := &\int_Y \Big ( \nabla_y^\top \phi \mathbf L_n \nabla_y \psi + \nabla_y^\top \phi \text{ } g(t,y) \psi + \phi (\partial_{y_2}(u(t,y) + w(t,y)) + \lambda)\psi  \text{ } dy.
			\end{split}
		\end{equation}
		
		where $g(t,y):= (y_2, u(t,y)  + w(t,y))^\top$.
		We have that:
		
		\begin{equation}
			\begin{split}
				|B[\phi,\psi;\lambda]| \leq & \lambda_{max}(\mathbf L_n)||\nabla_y \phi||_{\xLtwo(Y;\mathbb R^2)}||\nabla_y \psi||_{\xLtwo(Y;\mathbb R^2)} + ... \\ 
				&(s_{max} + u_{max} + w_{max}) ||\nabla_y \phi||_{\xLtwo(Y;\mathbb R^2)}||\psi||_{\xLtwo(Y;\mathbb R)} + ... \\
				&(\lambda + ||\partial_{y_2}(u(t,\cdot) + w(t,\cdot))||_{\xCzero_b(\text{cl }Y;\mathbb R)}) ||\phi||_{\xLtwo(Y;\mathbb R)}||\psi||_{\xLtwo(Y;\mathbb R)} \\
				&\leq C ||\phi||_{\xHone(Y;\mathbb R)} ||\psi||_{\xHone(Y;\mathbb R)}.
			\end{split}
		\end{equation}
		
		\noindent
		Also using (\ref{div_theorem}) (noting that $\lambda_{min}(\mathbf L_n) = \frac{1}{n}$)
		\[
		B[\phi,\phi;\lambda] \geq \lambda_{min}(\mathbf L_n) ||\nabla_y \phi||_{\xLtwo(Y;\mathbb R^2)}^2 + \mu ||\phi||^2_{\xLtwo(Y;\mathbb R)} = 
		\frac{1}{n} ||\nabla_y \phi||_{\xLtwo(Y;\mathbb R^2)}^2 + \mu ||\phi||^2_{\xLtwo(Y;\mathbb R)}
		\]
		where $\mu := \text{ess} \text{inf}_{y \in Y}[\lambda + \partial_{y_2}(u(t,y) + u(t,y)) - \frac{1}{2}\nabla_y \cdot g(t,y)]$. By our choice of $\lambda,$ $\mu > 0$. Thus, let $C:= \min \{ \frac{1}{n}, \mu\}$, then:
		\[
		B[\phi,\phi;\lambda] \geq C ||\phi||_{\xHone(Y;\mathbb R)}^2.
		\]
		So, the bilinear form $B[\cdot,\cdot]$ is coercive over $\xHone(Y;\mathbb R) \times \xHone(Y;\mathbb R)$. By the Lax-Milgram theorem, the variational problem:
		\[
		B[\phi,\psi;\lambda] = \langle f,\psi \rangle_{\xLtwo(Y;\mathbb R)}, \psi \in \xHone(Y;\mathbb R)
		\]
		has a unique solution $\phi \in \xHone(Y;\mathbb R)$ for every $f \in \xLtwo(Y;\mathbb R)$. By elliptic regularity (Theorem 31.27 of \cite{ern2021finite}), it follows that $\phi \in \xHtwo(Y;\mathbb R)$. So, it follows that for $\lambda \geq \omega$ as chosen, the operator $\lambda \mathcal I - \mathcal L_n(t;u,w)$ is invertible. Thus, the resolvent $\mathcal R(\lambda;\mathcal L_n(t;u,w))$ is also invertible for these $\lambda \geq \omega$.
		\newline
		\textbf{Step 4. Energy Estimates for $\phi$.}
		
		Following the proof of Theorem 6.2.2. of \cite{evans2010partial}, we have that:
		\[
		||\nabla_y \phi||_{\xLtwo(Y;\mathbb R)}^2 \leq B[\phi,\phi;\lambda] + C  ||\phi||_{\xLtwo(Y;\mathbb R)}^2.
		\]
		for some constant $C > 0$. If from before $\omega > C$, then:
		\[
		||\nabla_y \phi||_{\xLtwo(Y;\mathbb R)}^2 \leq B[\phi,\phi;\lambda] + \omega||\phi||_{\xLtwo(Y;\mathbb R)}^2.
		\]
		If not, simply re-select $\omega = C$, and $\mathcal R(\lambda;\mathcal L_n(t;u,w))$ will be invertible for $\lambda \geq \omega$, and the same conclusion as above follows.
		\newline
		\textbf{Step 5. Bounding the Resolvent.}
		
		Take the weak formulation of the previous elliptic problem (\ref{elliptic}):
		\[
		B[\phi,\psi;\lambda] = \langle f,\psi \rangle_{\xLtwo(Y;\mathbb R)}
		\]
		which is solved by $\phi \in D(\mathcal L_n)$ for every $\psi \in H^1(Y;\mathbb R)$ by Lax-Milgram and and  elliptic regularity
		with $\lambda > \omega$. From the energy estimate and letting $\psi = \phi$,
		\[
		||\nabla_y \phi||_{\xLtwo(Y;\mathbb R)}^2 + (\lambda - \omega)  \langle \phi,\phi \rangle_{\xLtwo(Y;\mathbb R)} \leq \langle f,\phi \rangle_{\xLtwo(Y;\mathbb R)}.
		\]
		So,
		\[
		(\lambda - \omega)\langle \phi,\phi \rangle_{\xLtwo(Y;\mathbb R)} \leq \langle f,\phi  \rangle_{\xLtwo(Y;\mathbb R)} \leq ||f||_{\xLtwo(Y;\mathbb R)}||\phi ||_{\xLtwo(Y;\mathbb R)}
		\]
		Now, since $\phi $ is a weak solution to the elliptic problem (\ref{elliptic}): 
		\[
		\phi  = \mathcal R(\lambda;\mathcal L_n(t;u,w))f
		\] 
		with $\lambda > \omega$. Thus:
		\[
		(\lambda - \omega)||\mathcal R(\lambda;\mathcal L_n(t;u,w))f||_{\xLtwo(Y;\mathbb R)}^2 \leq  ||f||_{\xLtwo(Y;\mathbb R)}||\mathcal R(\lambda;\mathcal L_n(t;u,w))f||_{\xLtwo(Y;\mathbb R)},
		\]
		so:
		\[
		||\mathcal R(\lambda;\mathcal L_n(t;u,w))||_{\mathcal L(\xLtwo(Y;\mathbb R);\xLtwo(Y;\mathbb R))} \leq \frac{1}{\lambda - \omega}.
		\]
		
		So, by Corollary 3.8. of \cite{pazy1983semigroups}, a version of the Hille-Yosida theorem, $\mathcal L_n(t;u,w)$ generates an $\omega$-contractive semigroup over $\xLtwo(Y;\mathbb R)$. Since this holds for every $t \in [0,T]$, this completes the proof of the original claim.
	\end{proof}
	\begin{lemma}
		Define:
		\[
		(\phi, (u, w)) \in B(0,\mathfrak R) \subset \Big ( \xLtwo(Y;\mathbb R) \times \xCzero([0,T] \times Y;\mathbf U \times \mathbf W) \Big )=:X,
		\] 
		
		\noindent
		$\mathfrak R > 0$ with the norm:
		\[
		||(\phi,(u,w))||_X:= ||\phi||_{\xLtwo (Y;\mathbb R)} + \sup_{(t,y) \in [0,T]\times Y}||(u(t,y),w(t,y)) ||_{\mathbb R^2}.
		\]
		The mapping:
		$
		( \cdot,-\mathcal A^*(t;\cdot,\cdot)\psi )_{\mathcal D^*(Y;\mathbb R),\mathcal D(Y;\mathbb R)} : X \rightarrow \mathbb R
		$
		\begin{equation*}
			\begin{split}
				&( \phi ,-\mathcal A^*(t;u,w)\psi )_{\mathcal D^*(Y;\mathbb R),\mathcal D(Y;\mathbb R)} := \int_Y \epsilon \phi \partial_{y_2}^2 \psi  + y_2 \phi \partial_{y_1} \psi + (u + w)\phi \partial_{y_2} \psi \text{ } dy
			\end{split} 
		\end{equation*}
		for fixed $\psi \in \mathcal D(Y;\mathbb R)$ is locally Lipschitz in $X$.

	\end{lemma}
	\begin{proof}
		First, recall that the finite sum of Lipschitz functions is also Lipschitz. Note that the pairing given is a sum of three terms, the first two of which are bounded and linear in $\phi$ and constant w.r.t. $(u,w)$, and the final term is bilinear in $(\phi,(u,w))$. The first two terms are trivially Lipschitz in $\{X, ||\cdot||_X\}$. As for the last term, let 
		\begin{equation}
			B[\phi,(u,w)]:= \int_Y (u + w)\phi \partial_{y_2} \psi \text{ } dy.
		\end{equation}
		
		Then, we have
		
		\begin{equation*}
			\begin{split}
				&|B[\phi_1, (u_1,w_1)] - B[\phi_2, (u_2,w_2)]| \leq |B[\phi_1,(u_1,w_1)] - B[\phi_1,(u_2,w_2)]| + ... \\
				&|B[\phi_1, (u_2,w_2)] - B[\phi_2,(u_2,w_2)]| = ... \\ 
				& |B[\phi_1, (u_1 - u_2, w_1 - w_2)]| + |B[\phi_1 - \phi_2, (u_2,w_2)]| \leq ... \\ 
				&\Big (||\phi_1||_{\xLtwo(Y;\mathbb R)} ||(u_1-u_2,w_1-w_2)||_{\xCzero([0,T] \times Y;\mathbf U \times \mathbf W)} + ... \\ 
				&||\phi_1 - \phi_2||_{L^2(Y;\mathbb R)}||(u_2,w_2)||_{\xCzero([0,T] \times Y;\mathbf U \times \mathbf W)}\Big ) ||\psi||_{\xCone(Y;\mathbb R)}.
			\end{split} 
		\end{equation*}
		
		\noindent Now, let 
		\[
		L:= \max \{\epsilon ||\psi||_{\xCtwo(Y;\mathbb R)},s_{max}||\psi||_{\xCone(Y;\mathbb R)}, ||\psi||_{\xCone(Y;\mathbb R)} \max \{ \mathfrak R, u_{max} + w_{max}\} \}.
		\] 
		Then, 
		\[
		|B[\phi_1, (u_1,w_2)] - B[\phi_2, (u_2,w_2)]| \leq L||(\phi_1,u_1,w_1) - (\phi_2, u_2, w_2)||_X
		\]
		so $B[\cdot,\cdot]$ is Lipschitz in $\{ X, ||\cdot||_X\}$. 
	\end{proof}
	With these preliminary results completed, we now return to the task of proving Theorem 1. We define sequences of approximate solutions $\{(A,B)^n(\cdot) \}_{n=1}^\infty, (A,B)^n(\cdot): [0,T] \rightarrow \mathbb R^{K \times K} \times \mathbb R^{K \times K}$, and 
	$\{\rho^n(\cdot) \}_{n=0}^\infty, \rho^n : [0,T] \times Y \rightarrow \mathbb R.$ 
	Let:
	\[
	\rho^0(t,\cdot) = \rho_0 \text{ } \forall \text{ } t \in [0,T].
	\]
	\textbf{Step 1. First Pass for the Weights.}

	Let $(A,B)^1(\cdot)$ be subject to the ODEs:
	
	\begin{equation}
		\begin{split}
			&\xDrv  {A^1(t)}{t} = \mathfrak G(A^{1}(t), B^{1}(t), \rho^{0}(t)), \text{ }A^1(0) = A_{0}, \\
			&\xDrv  {B^1(t)}{t} = \mathfrak H(A^{1}(t), B^{1}(t), \rho^{0}(t)), \text{ }B^1(0) = B_{0}.
		\end{split} 
	\end{equation}
	
	Lemma 1 gives local Lipschitzianity of the vector fields $\mathfrak G(\cdot, \rho^0(t)), \mathfrak H(\cdot,\rho^0(t))$ in $A,B$. So, by Picard-Lindel\"of, there is a unique local classical solution to this ODE system of type $\xCone([0,t^*(A_0,B_0));\mathbb R^{K \times K} \times \mathbb R^{K \times K})$. Note from (\ref{bounds}), we have that (using real constant $C$ generically):
	\[
	\mathfrak G(A,B,\rho) \leq C(1 + \frac{1}{C}(||A||_F + ||B||_F))
	\]
	for $C > 0$, and similarly for $\mathfrak H$. So, from the Gr\"onwall-Bellman inequality, for fixed $s,t \in \mathbb R_0^+, 0 \leq s \leq t$:
	
	\begin{equation*}
		\begin{split}
			&||A^1(t)||_F + ||B^1(t)||_F \leq  ||A^1(s)||_F + ||B^1(s)|| + ... \\ 
			&  \int_s^t ||\mathfrak G(A^1(\tau),B^1(\tau), \rho^{0}(t))||_F +  ||\mathfrak H(A^1(\tau),B^1(\tau), \rho^{0}(t))||_F \text{ } d\tau \leq ...  \\
			&  ||A^1(s)||_F + ||B^1(s)|| + 2 \int_s^t C(1 + \frac{1}{C}(||A^1(\tau)||_F + ||B^1(\tau)||_F)) \text{ } d \tau,
		\end{split} 
	\end{equation*}
	so (using constant $C$ generically):
	\begin{equation*}
		||A^1(t)||_F + ||B^1(t)||_F \leq (||A^1(s)||_F + ||B^1(s)||_F + C(t-s))\exp(t-s).
	\end{equation*}
	
	This rules out finite-time blowup, so there is a unique global classical solution to the ODE system of type
	$(A,B)^1 \in \xCone([0,T];\mathbb R^{K \times K} \times \mathbb R^{K \times K})$. 
	Define for $n \geq 1$:
	
	\begin{equation*}
		\begin{split}
			& \tilde{\mathcal V}^{A^n,B^n}(t,y) := \sum_{i=0}^{K-1} \sum_{j=0}^{K-1} \big ( a^n_{ij}(t) \sin(\frac{2 \pi i y_1}{\mathfrak L}) + b^n_{ij}(t)\cos(\frac{2 \pi i y_1}{\mathfrak L}) \big) \cos(\frac{2\pi j  y_2}{s_{max}} ),
		\end{split}
	\end{equation*}
	
	\[
	\tilde u^n(\cdot) := \tilde u(\partial_{y_2} \tilde{\mathcal V}^{A^n,B^n}(\cdot)), 
	\tilde w^n(\cdot) := \tilde w(\partial_{y_2} \tilde{\mathcal V}^{A^n,B^n}(\cdot)).
	\]
	Note that our choice of $\tilde u^n, \tilde w^n$ satisfy the assumptions we needed for $u,w$ in Lemma 3.
	
	\noindent
	\textbf{Step 2. First Pass for the Density.}
	
	We seek a classical solution to abstract problem:
	\begin{equation}
		\begin{split}
			\label{firstone}
			&\xDrv  {\rho^1(t)}{t} =  \mathcal L_1(t;\tilde u^1(\cdot),\tilde w^1(\cdot)) \rho^1(t) \text{ in }   (0, T] \times \text{int } Y; \\
			& \rho^1(0) = \rho_0 \in D(\mathcal L_1). \\
		\end{split} 
	\end{equation}
	We present Theorem 5.5.3. from \cite{pazy1983semigroups} to aid us now:
	\begin{theorem}
		(\textbf{Pazy})
		Let $U$ be a Banach space, and $\{ \mathbf A(t)\}_{t \in [0,T]}$ be a stable family of generators of $C_0$-semigroups on $U$ s.t. $\mathbf A(t): D \subset U \rightarrow U$ for each $t \in [0,T]$, and for each $f \in D$, $\mathbf A(t)f \in C^1([0,T];U)$. If $g \in C^1([0,T];U)$, then for every $x_0 \in D$, the abstract Cauchy problem:
		\[
		\xDrv  {x(t)}{t} = \mathbf A(t)x(t) + g(t), \text{ }x(0) = x_0
		\]
		has a unique classical solution in the sense of $U$: 
		\[
		x \in \xCzero([0,T];D) \cap \xCone([0,T];U)
		\] 
		where for each $t \in [0,T]$:
		\[
		x(t) = \mathbf \Phi_{\mathbf A}(t,0)x_0 + \int_{0}^t \mathbf \Phi_{\mathbf A}(t,s) g(s)  ds,
		\]
		where $\mathbf \Phi_{\mathbf A}(\cdot,\cdot) : [0,T] \times [0,T] \rightarrow \mathcal L(D;D)$ is the evolution system \cite{pazy1983semigroups} associated to 
		$\{\mathbf A(t)\}_{t \in [0,T]}$.
		If $\mathbf A(t)$ each generate $\omega$-contractive $C_0$ semigroups, then:
		\[
		||\mathbf \Phi_{\mathbf A}(t,s)||_{\mathcal L(D;D)} \leq  e^{\omega (t-s)}.
		\]
	\end{theorem}
	This result is the infinite-dimensional analogue of the usual finite-dimensional LTV systems theory \cite{brockett2015finite}.
	
	Let $D:=  D(\mathcal L_1), U:=\xLtwo(Y;\mathbb R) $.  We have from Lemma 2 for $n \geq 1$ that 
	$\mathcal L_n(t;u,w)f \in \xCone([0,T];U)$ for each $f \in D$. 
	Now, we invoke Lemma 3 and Theorem 2, and conclude that: 
	\[
	\rho^1(\cdot) \in \xCzero([0,T];D(\mathcal L_1)) \cap \xCone([0,T]; \xLtwo (Y;\mathbb R))
	\]
	uniquely solves the previously given PDE in the sense of $\xLtwo(Y;\mathbb R)$. Note that 
	$D(\mathcal L_n) = D(\mathcal A) = D$ for every $n \geq 1$ due to the structure of $\partial Y$ as in Remark \ref{remark1}.
	
	\noindent
	\newline
	\newline
	\newline
	\newline
	\newline
	\newline
	\textbf{Step 3. Iteration of Solutions.}
	
	Now, for $n \geq 2$, let:
	\begin{equation}
		\begin{split}
			\label{iterates}
			&\xDrv  {A^n(t)}{t}  = \mathfrak G(A^{n}(t), B^{n}(t),\rho^{n-1}(t)), \text{ }A^{n}(0) = A_{0}; \\
			&\xDrv  {B^n(t)}{t} = \mathfrak H(A^{n}(t), B^{n}(t),\rho^{n-1}(t)), \text{ }B^{n}(0) = B_{0}; \\
			&\xDrv  {\rho^n(t)}{t}  =   \mathcal L_n(t;\tilde u^n(\cdot),\tilde w^n(\cdot)) \rho^n(t) \text{ in }  (0, T] \times \mathcal Y; \\
			& \rho^n(0,\cdot) = \rho_0 \in D(\mathcal L_n).
		\end{split} 
	\end{equation}
	
	First, the ODEs are solved and then PDE, as the coupling between the equations is only one-directional. For the ODE, we note that by continuous differentiability in time of $\rho^{n-1}$, $\mathfrak G(A(\cdot), B(\cdot), \rho^{n-1}(\cdot))$ will also be continuously differentiable in time. By Picard-Lindel\"of there is a local classical solution on $[0,t^{*,n}(A_0,B_0))$, and then we apply our approach to rule out finite-time blowup via Gr\"onwall-Bellman from Section 1 to conclude that there is a classical solution $(A,B)^n(t) \in \xCone([0,T];\mathbb R^{K \times K} \times \mathbb R^{K \times K}).$ Then, we employ the argument of Section 2 to conclude that for each $n \geq 2$, there is a classical solution to the PDE in this iterated system of type:
	\[
	(A, B, \rho(\cdot))^n(\cdot) \in \xCzero([0,T];D(\mathcal L_n) \times \mathbb R^{K \times K} \times \mathbb R^{K \times K}) \cap \xCone([0,T];\mathbb H^{0}).
	\]
	\noindent
	\textbf{Step 4. Boundedness of Approximate Solutions.}
	
	Now, we must show that  there is a limit point of this sequence. We will briefly comment on boundedness of the approximate solutions now. From the bounds given by Pazy in \cite{pazy1983semigroups}, and from a $\omega-$contractivity of $\mathcal L_n(t;\tilde u^n(\cdot),\tilde w^n(\cdot))$,  we know that (using $C$ as an arbitrary non-negative constant):
	\begin{equation}
		\label{ball}
		||\rho^n(t)||_{D(\mathcal L_n)} \leq e^{\omega t} ||\rho_0||_{D(\mathcal L_n)} \leq e^{\omega T} ||\rho_0||_{D(\mathcal L_n)}\leq C ||\rho_0||_{\xHtwo(Y;\mathbb R)} =: \mathfrak R,
	\end{equation}
	\noindent
	since $D(\mathcal L_n) \subset \xHtwo(Y;\mathbb R)$.
	As for $A^n,B^n$, we have from the proof of Lemma 1:
	\begin{equation*}
		\begin{split}
			&||A^n(t)||_{F} + ||B^n(t)||_{F} \leq ... \\ 
			& ||A_0||_{F} + ||B_0||_{F} + \int_{0}^t || \mathfrak G(A^{n}(s), B^{n}(s), \rho^{n-1}(s)) ||_F ds + ... \\
			&\int_{0}^t ||\mathfrak H(A^{n}(s), B^{n}(s), \rho^{n-1}(s))||_F \text{ } ds \leq ... \\ 
			&||A_0||_{F} + ||B_0||_{F} + Ct +
			\int_{0}^t ||A^n(s)||_F + ||B^n(s)||_F \text{ } ds
		\end{split} 
	\end{equation*}
	\noindent
	again using our result from (\ref{bounds}).
	Now, by the Gr\"onwall-Bellman \cite{evans2010partial} inequality:
	
	\begin{equation*}
		\begin{split}
			&||A^n(t)||_{F} + ||B^n(t)||_{F} \leq (||A_0||_{F} + ||B_0||_{F} + CT)e^{T}=: \mathfrak D,
		\end{split} 
	\end{equation*}
	\noindent
	hence $(A, B, \rho)^n(\cdot)$ is uniformly bounded in the usual supremum-norm topology on 
	$C([0,T]; \mathbb H^{2})$. We also know that our classical solution $\rho^n(\cdot)$ is s.t.:
	\begin{equation}
		\xDrv  {\rho^n(t)}{t} = \mathcal L_n(t;u^n(t), w^n(t))\rho^n(t), 
	\end{equation}
	in $\xLtwo(Y;\mathbb R)$ and from (\ref{ball}):
	\[
	||\xDrv  {\rho^n(t)}{t}||_{\xLtwo(Y;\mathbb R)} = ||\mathcal L_n(t;u^n(t), w^n(t))\rho^n(t)||_{\xLtwo(Y;\mathbb R)} \leq ||\rho^n(t)||_{D(\mathcal L_n)} \leq \mathfrak R.
	\]
	Similarly:
	\begin{equation}
		\xDrv  {A^n(t)}{t}  = \mathfrak G(A^n(t), B^n(t), \rho^{n-1}(t));
	\end{equation}
	\begin{equation}
		\xDrv  {B^n(t)}{t} = \mathfrak H(A^n(t), B^n(t), \rho^{n-1}(t));
	\end{equation}
	and from earlier:
	\begin{equation}
		||\xDrv  {A^n(t)}{t}||_F \leq C(1 + \frac{1}{C}(||A^n(t)||_F + ||B^n(t)||_F)) \leq \mathfrak D;
	\end{equation}
	\begin{equation}
		||\xDrv  {B^n(t)}{t}||_F \leq C(1 + \frac{1}{C}(||A^n(t)||_F + ||B^n(t)||_F)) \leq \mathfrak D.
	\end{equation}
	So, $\xDrv  {(A^n, B^n, \rho^n)}{t}$ is uniformly bounded in $\xLinfty(0,T;\mathbb H^0)$.
	
	\noindent
	\textbf{Step 5. Compact Embeddings and Strong Compactness.}
	
	\noindent
	From \cite{simon1986compact}, we have:
	\begin{theorem} \textbf{(Aubin-Lions-Simon)}
		Let $X_0, X_1, X_2$ be Banach spaces with $X_0 \subseteq X_1 \subseteq X_2$, and $X_0$ compactly embedded in $X_1$, $X_1$ continuously embedded in $X_2$. For $1 \leq q, p \leq \infty$ define:
		\[
		S := \{f \in \xLn{q}(0,T;X_0) : \xDrv  {f}{t} \in \xLn{p}(0,T;X_2)  \}.
		\]	
		\begin{enumerate}
			\item If $q < \infty$, then the embedding of $S$ into $\xLn{p}(0,T;X_1)$ is compact; and
			\item if $q = \infty$, and $p > 1$, then the embedding of $S$ into $\xCzero([0,T];X_1)$ is compact.
		\end{enumerate}
	\end{theorem}

	Now, $\xHn{{p+1}}$ is compactly embedded in $\xHn{p}$ by Rellich-Kondrachov \cite{evans2010partial, gilbarg2015elliptic, brezis2011functional}, and $\xHn{p}$ is likewise compactly, hence continuously, embedded into $\xHn{{p-1}}$. For the embeddings of the finite dimensional parts of $\mathbb H^{(\cdot)}$, recall the Heine-Borel theorem \cite{kolmogorov1975introductory}. Let $X_0:= \mathbb H^2, X_1:= \mathbb H^1, X_2:= \mathbb H^0$. We have from earlier that $\{(A, B, \rho)^{n}(\cdot)\}_{n=1}^\infty$ is bounded in $\xCzero([0,T]; \mathbb H^{2})$, hence in $\xLinfty(0,T;\mathbb H^2)$, and that $\{\xDrv  {(A, B, \rho)^{n}}{t} \}_{n=1}^\infty$ is bounded in $\xLinfty(0,T;\mathbb H^0)$. Thus, $\{(A, B,\rho)^{n}(\cdot)\}_{n=1}^\infty$ is bounded in $S$ under the previous definitions of  $X_0,X_1,X_2$. So, by Theorem 3, there is some $(A, B, \rho)^{*}(\cdot) \in \xCzero([0,T];\mathbb H^1)$ s.t.:
	\[
	(A, B, \rho)^{n}(\cdot) \rightarrow (A, B, \rho)^{*}(\cdot) \text{ strongly in } \xCzero([0,T];\mathbb H^1).  
	\] 
	
	\noindent \textbf{Step 6. Existence of Weak Solutions to ADP ODE-PDE System.}
	Finally, it remains to show that $(A, B, \rho)^{*}(\cdot)$ solve the ADP system (\ref{learning}) in the sense of (\ref{def_1} - \ref{def_4}). From the local Lipschitzianity of $\mathfrak G$ from Lemma 1:
	
	\begin{equation}
		\begin{split}
			&|| \int_0^t  \mathfrak G(A^n(s), B^n(s), \rho^{n-1}(s))- \mathfrak G(A^*(s), B^*(s), \rho^*(s)) \text{ } ds ||_{\mathbb R^{K \times K}} \leq ... \\ 
			& \theta^{-1} L \int_0^t ||(A^n, B^n, \rho^{n-1})(s) - (\rho,A,B)^*(s)||_{\mathbb H^0} \text{ } ds \rightarrow 0
		\end{split} 
	\end{equation}
	where $L:= \max \{ L_1, L_2\}$.
	and similarly for $\mathfrak H$ using the convergence we showed in Step 5. So, (\ref{def_1}, \ref{def_2}) are satisfied.
	Let for $\phi \in \xLtwo(Y;\mathbb R), \psi \in \mathcal D(Y;\mathbb R)$:
	
	\begin{equation*}
		\begin{split}
			&(\phi, \mathcal L^*_n(t;u,w)\psi )_{\mathcal D^*(Y;\mathbb R), \mathcal D(Y;\mathbb R)}:= ... \\ 
			& ( \phi, -\mathcal A^*(t;u,w)\psi )_{\mathcal D^*(Y;\mathbb R), \mathcal D(Y;\mathbb R)} + \int_Y \phi \nabla_y \cdot [\mathbf L_n \nabla_y \psi] \text{ } dy
		\end{split} 
	\end{equation*}
	with $\mathcal L^*_n(t;u,w)$ the formal adjoint of $\mathcal L_n(t;u,w)$.
	\noindent
	Since $\rho^n(\cdot)$ is a classical solution to the abstract Cauchy problems we formulated in (\ref{firstone}) and (\ref{iterates}) for each $n \geq 1$, it follows that it is also a weak solution:
	
	\begin{equation*}
		\begin{split}
			&( \rho^n(t), \psi )_{\mathcal D^*(Y;\mathbb R), \mathcal D(Y;\mathbb R)} = ... \\ 
			&( \rho_0, \psi )_{\mathcal D^*(Y;\mathbb R), \mathcal D(Y;\mathbb R)} + \int_0^t ( \rho(s), \mathcal L^*_n(s;\tilde u^n(\partial_{y_2} \tilde{\mathcal V^n}),\tilde w^n(\partial_{y_2}\tilde{\mathcal V}^n) ) \psi )_{\mathcal D^*(Y;\mathbb R), \mathcal D(Y;\mathbb R)} \text{ } ds
		\end{split} 
	\end{equation*}
	
	\noindent for arbitrary $\psi \in \mathcal D(Y;\mathbb R)$. It also follows from Lemma 4, and the smoothness of $\tilde u(\partial_{y_2} \tilde{\mathcal V}^{A,B}),\tilde w(\partial_{y_2} \tilde{\mathcal V}^{A,B})$ w.r.t. $A,B$ that:
	
	\begin{equation*}
		\begin{split}
			&|\int_0^t ( \rho^n(s), \mathcal L^*_n(s;\tilde u(\partial_{y_{2}} \tilde{\mathcal V}^n),\tilde w(\partial_{y_{2}} \tilde{\mathcal V}^n) ) \psi )_{\mathcal D^*(Y;\mathbb R), \mathcal D(Y;\mathbb R)} - ... \\
			&( \rho^*(s), -\mathcal A^*(s;\tilde u(\partial_{y_{2}} \tilde{\mathcal V}^*),\tilde w(\partial_{y_{2}} \tilde{\mathcal V}^*) ) \psi )_{\mathcal D^*(Y;\mathbb R), \mathcal D(Y;\mathbb R)} \text{ } ds | + ... \\ 
			&\int_0^t \frac{1}{n}| ( \rho^n(s),  \nabla_y^2 \psi )_{\mathcal D^*(Y;\mathbb R), \mathcal D(Y;\mathbb R)} | \leq ... \\ 
			&\int_0^t\Big (C||(\rho,A,B)^n(s)-(\rho,A,B)^*(s)||_{\mathbb H^0} \Big ) \text{ } ds + ... \\ 
			&\int_0^t \frac{1}{n}| ( \rho^n(s),  \nabla_y^2 \psi )_{\mathcal D^*(Y;\mathbb R), \mathcal D(Y;\mathbb R)} | \text{ } ds.
		\end{split} 
	\end{equation*}
	\noindent
	Due to the compactness we showed earlier, and since $||f||_{\xHone(Y;\mathbb R)} = (||\nabla_{y}f||_{\xLtwo(Y;\mathbb R^2)}^2 + ||f||_{\xLtwo(Y;\mathbb R)}^2)^{1/2}$, the first term in the last line goes to 0 as $n \rightarrow \infty$. As for the second:
	\[
	\frac{1}{n} |( \rho^n(s),  \nabla_y^2 \psi )_{\mathcal D^*(Y;\mathbb R), \mathcal D(Y;\mathbb R)} | \leq \frac{1}{n} \mathfrak R  ||\psi||_{\xHtwo(Y;\mathbb R)} \leq \mathfrak R  ||\psi||_{\xHtwo(Y;\mathbb R)}, n \geq 1 .
	\]
	Now, use $ \mathfrak R  ||\psi||_{\xHtwo(Y;\mathbb R)}$ to dominate
	$| \frac{1}{n} ( \rho^n(s),  \nabla_y^2 \psi )_{\mathcal D^*(Y;\mathbb R), \mathcal D(Y;\mathbb R)}|$, and by the dominated convergence theorem:
	\[
	\begin{split}
	&\lim_{n\rightarrow \infty} \int_{0}^t | \frac{1}{n} ( \rho^n(s),  \nabla_y^2 \psi )_{\mathcal D^*(Y;\mathbb R), \mathcal D(Y;\mathbb R)} | \text{ }ds 
	= ... \\ 
	&\int_{0}^t \lim_{n\rightarrow \infty} \frac{1}{n}| ( \rho^n(s),  \nabla_y^2 \psi )_{\mathcal D^*(Y;\mathbb R), \mathcal D(Y;\mathbb R)}| \text{ } ds.
	\end{split}
	\]
	Also:
	\[
	0 \leq 	\frac{1}{n} |( \rho^n(s),  \nabla_y^2 \psi )_{\mathcal D^*(Y;\mathbb R), \mathcal D(Y;\mathbb R)} | \leq \frac{1}{n} \mathfrak R  ||\psi||_{\xHtwo(Y;\mathbb R)},
	\]
	so by the squeeze theorem,
	\[
	\lim_{n\rightarrow \infty} \frac{1}{n}| (  \rho^n(s),  \nabla_y^2 \psi )_{\mathcal D^*(Y;\mathbb R), \mathcal D(Y;\mathbb R)} | = 0.
	\]
	So, (\ref{def_3}) is satisfied.
	Finally, we move onto the boundary condition. Define:
	\[
	b_n:= (\nabla_y^\top \rho_n) \mathbf L_n \mathbf n |_{\partial Y}, b^*:= \lim_{n \rightarrow \infty }(\nabla_y^\top \rho^n) \mathbf L_n \mathbf n|_{\partial Y} = (\nabla_y^\top \rho^*) \mathbf L \mathbf n|_{\partial Y}.
	\]
	From each solution $\rho^n$, we know that $b_n(t,y) = 0$ almost everywhere on $[0,T] \times \partial Y$. So:
	\[
	-e^{-n} < ||b_n||_{\xLtwo([0,T ] \times \partial Y;\mathbb R)}^2 < e^{-n},
	\]
	and by the squeeze theorem, $\lim_{n \rightarrow \infty }||b_n||_{\xLtwo([0,T ] \times \partial Y;\mathbb R)}^2 = 0$. Now, we can also dominate $b_n$ by $g(x) = 1$, so it follows:
	\[
	\lim_{n \rightarrow \infty} ||b_n||_{\xLtwo([0,T ] \times \partial Y;\mathbb R)}^2 =  ||b^*||_{\xLtwo([0,T ] \times \partial Y;\mathbb R)}^2 = 0,
	\]
	hence $b^* = 0$ a.e. on $[0,T] \times \partial Y$. So, (\ref{def_4}) is satisfied.

	Thus, the limit point $(A,B,\rho)^*$ is a weak solution to the ADP system. 
	
	\noindent
	\newline
	\textbf{Step 7. Uniqueness.}
	
	As to verify uniqueness, suppose there exists another weak solution $(\bar A, \bar B, \bar \rho) \in \xCzero([0,T];\mathbb H^1)$ to the ADP system which satisfies the previous definition in equations (\ref{def_1}) - (\ref{def_4}), $\bar A(0) = A_0, \bar B(0) = B_0, \bar \rho(0,\cdot) = \rho_0$. Then:
	\[
	||(A, B, \rho)^*(0) - (\bar A, \bar B,\bar \rho)(0)||_{\mathbb H^1} = 0.
	\] 
	We have from the definition of the weak solution (\ref{def_1} - \ref{def_4}) that:
	
	\begin{equation*}
		\begin{split}
			&||(A,B,\rho)^*(t) - (\bar A, \bar B,\bar \rho)(t)||_{\mathbb H^1} \leq  ... \\
			& \int_0^t || \mathfrak G(A^*(s), B^*(s), \rho^*(s))- \mathfrak G(\bar A(s), \bar B(s), \bar \rho(s))||_{\mathbb R^{K \times K}} ds  + ... \\ 
			& \int_0^t  ||\mathfrak H(A^*(s), B^*(s), \rho^*(s))- \mathfrak H(\bar A(s), \bar B(s), \bar \rho(s)) ||_{\mathbb R^{K \times K}} ds  + ... \\
			&\int_0^t |(\rho^*(s), -\mathcal A^*(s;\tilde u(\partial_{y_{2}} \tilde{\mathcal V}^*),\tilde w(\partial_{y_{2}} \tilde{\mathcal V}^*) ) \psi )_{\mathcal D^*(Y;\mathbb R), \mathcal D(Y;\mathbb R)} - ... \\
			&( \bar \rho(s), -\mathcal A^*(s;\tilde u(\partial_{y_{2}} \tilde{\mathcal V}^{\bar A,\bar B}),\tilde w(\partial_{y_{2}} \tilde{\mathcal V}^{\bar A,\bar B}) ) \psi )_{\mathcal D^*(Y;\mathbb R), \mathcal D(Y;\mathbb R)} | ds 
		\end{split} 
	\end{equation*}
	
	Again, from the local Lipschitzianity of $\mathfrak G$ in Lemma 1:
	
	\begin{equation*}
		\begin{split}
			&\int_0^t ||  \mathfrak G(A^*(s), B^*(s),\rho^*(s))-  \mathfrak G(\bar \rho(s), \tilde A(s), \tilde B(s))||_{\mathbb R^{K \times K}} ds  \leq ... \\ 
			&\theta^{-1} L\int_0^t ||(A,B,\rho)^*(s) - (\bar A, \bar B, \bar \rho)(s)||_{\mathbb H^0} ds \leq ... \\ 
			& \theta^{-1} L\int_0^t ||(A,B,\rho)^*(s) - (\bar A, \bar B, \bar \rho)(s)||_{\mathbb H^1} ds
		\end{split} 
	\end{equation*}
	
	and similarly for $\mathfrak H$. From Lemma 4,
	
	\begin{equation*}
		\begin{split}
			&\int_0^t |( \rho^*(s), \mathcal A^*(s;\tilde u(\partial_{y_{2}} \tilde{\mathcal V}^*),\tilde w(\partial_{y_{2}} \tilde{\mathcal V}^*) ) \psi )_{\mathcal D^*(Y;\mathbb R), \mathcal D(Y;\mathbb R)} - ... \\
			&( \bar \rho(s), \mathcal A^*(s;\tilde u(\partial_{y_{2}} \tilde{\mathcal V}^{\bar A,\bar B}),\tilde w(\partial_{y_{2}} \tilde{\mathcal V}^{\bar A,\bar B}) ) \psi )_{\mathcal D^*(Y;\mathbb R), \mathcal D(Y;\mathbb R)} |ds  \leq ...  \\ 
			& C\int_0^t ||(A,B,\rho)^*(s) - (\bar A, \bar B, \bar \rho)(s)||_{\mathbb H^0} ds \leq ... \\ 
			&C\int_0^t ||(A,B,\rho)^*(s) - (\bar A, \bar B, \bar \rho)(s)||_{\mathbb H^1} ds.
		\end{split} 
	\end{equation*}
	
	Thus:
	
	\begin{equation*}
		\begin{split}
			&||(A,B,\rho)^*(t) - (\bar A, \bar B, \bar \rho)(t)||_{\mathbb H^1} \leq ... \\ 
			& (2 \theta^{-1} L + C) \int_0^t ||(A,B,\rho)^*(s) - (\bar A, \bar B, \bar \rho)(s)||_{\mathbb H^1} ds.
		\end{split} 
	\end{equation*}
	
	Let $\beth:= (2 \theta^{-1} L + C)$. By the Gr\"onwall-Bellman inequality:
	$$
	||(A,B,\rho)^*(t) - (\bar A, \bar B, \bar \rho)(t)||_{\mathbb H^1} \leq e^{\beth t}||(A,B,\rho)^*(0) - (\bar A, \bar B, \bar \rho)(0)||_{\mathbb H^1} = 0.
	$$
	So, $(A,B,\rho)^*$ is the unique weak solution to the ADP system. Letting $(A,B,\rho) = (A,B,\rho)^*$ completes the proof of Theorem 1.
\end{proof}
\section{Numerical Results}
\begin{figure}
	\centering
	\includegraphics[scale=.6, trim={0 15 0 34}]{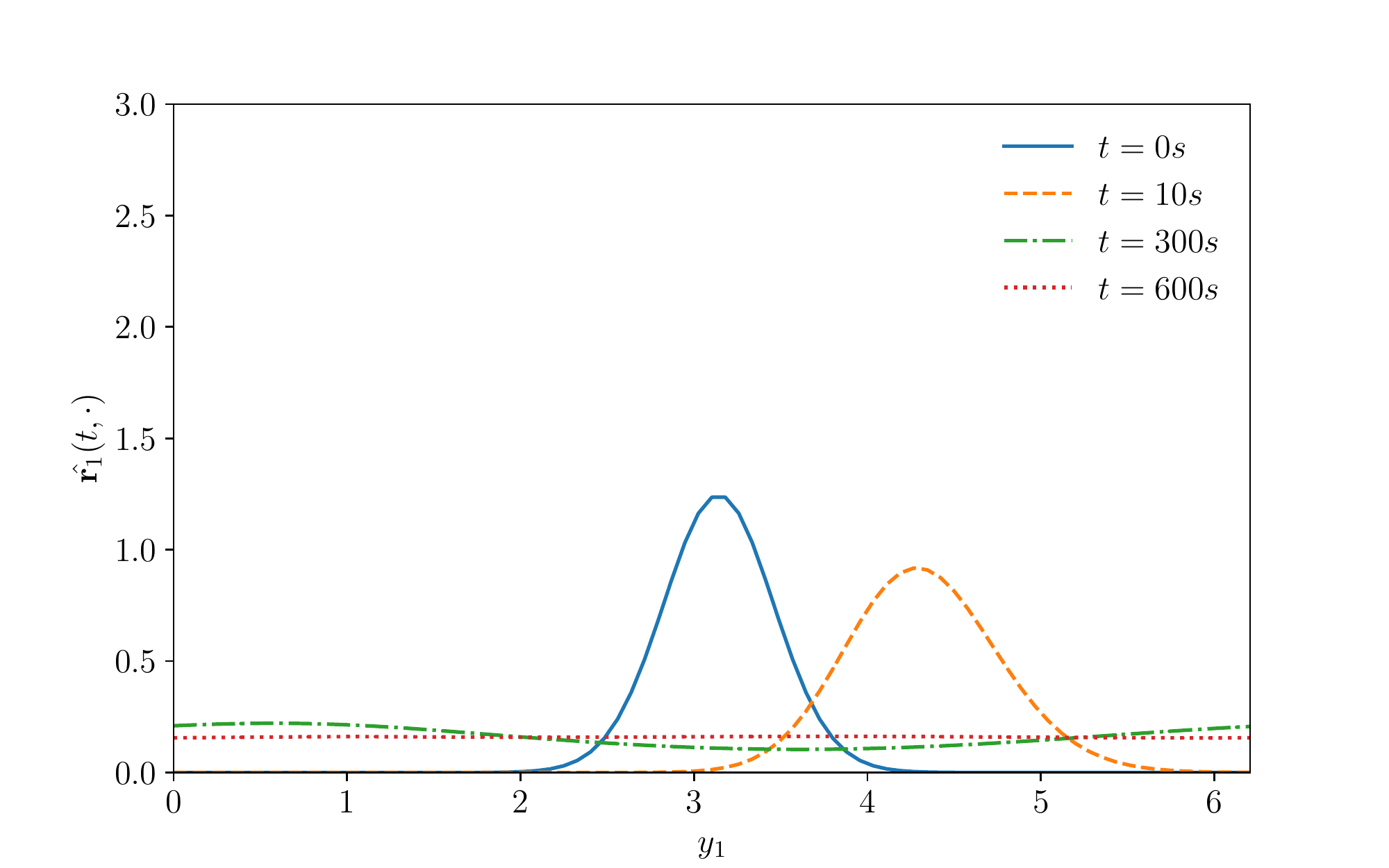}
	\centering
	\includegraphics[scale=.6, trim={0 -30 0 20}]{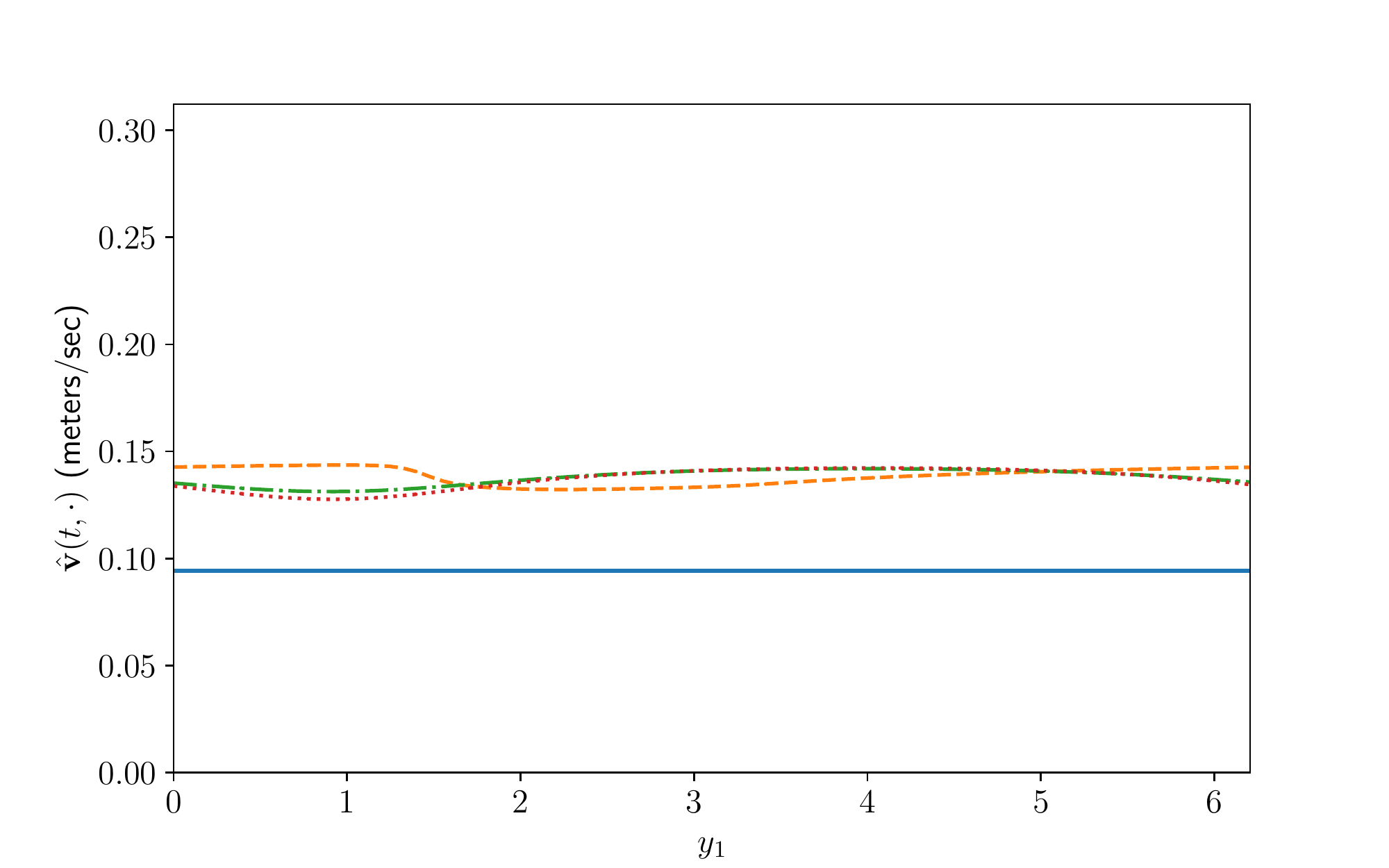}
	\caption{(top) The numerical spatial probability distribution of the traffic flow at different times as a function of space. The road is a closed loop, so the flow on the RHS comes back to the LHS. (bottom) The numerical bulk velocity of the traffic flow at different times as a function of space. (top) and (bottom) share a horizontal axis. At first, a there is a significant slowdown in the traffic flow, but congestion is rapidly dissipated. This slowdown region is later dissipated itself. }
\end{figure}
\begin{figure}
	\centering
	\includegraphics[scale=.6, trim={0 0 0 0}]{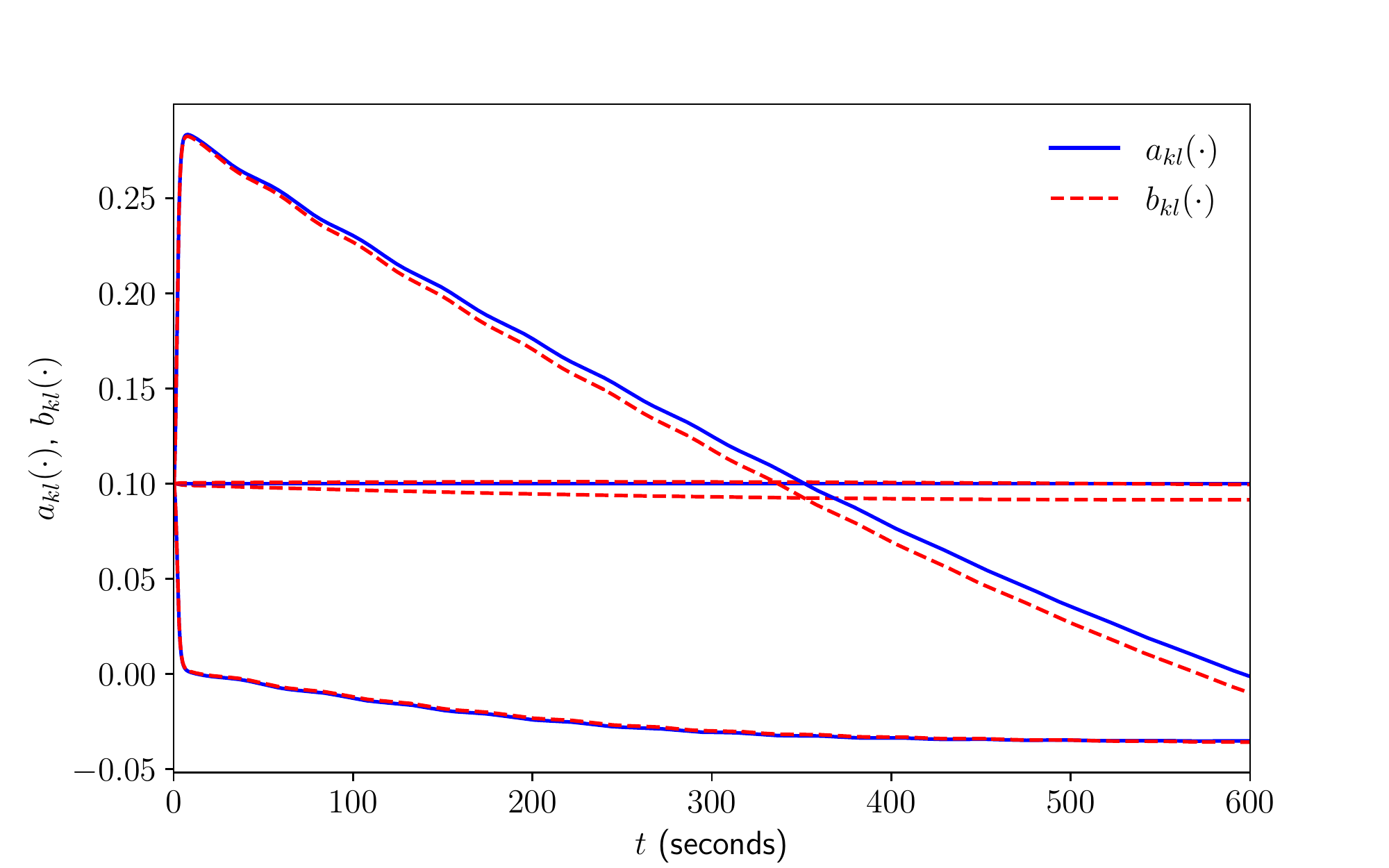}
	\centering
	\includegraphics[scale=.6, trim={0 0 0 0}]{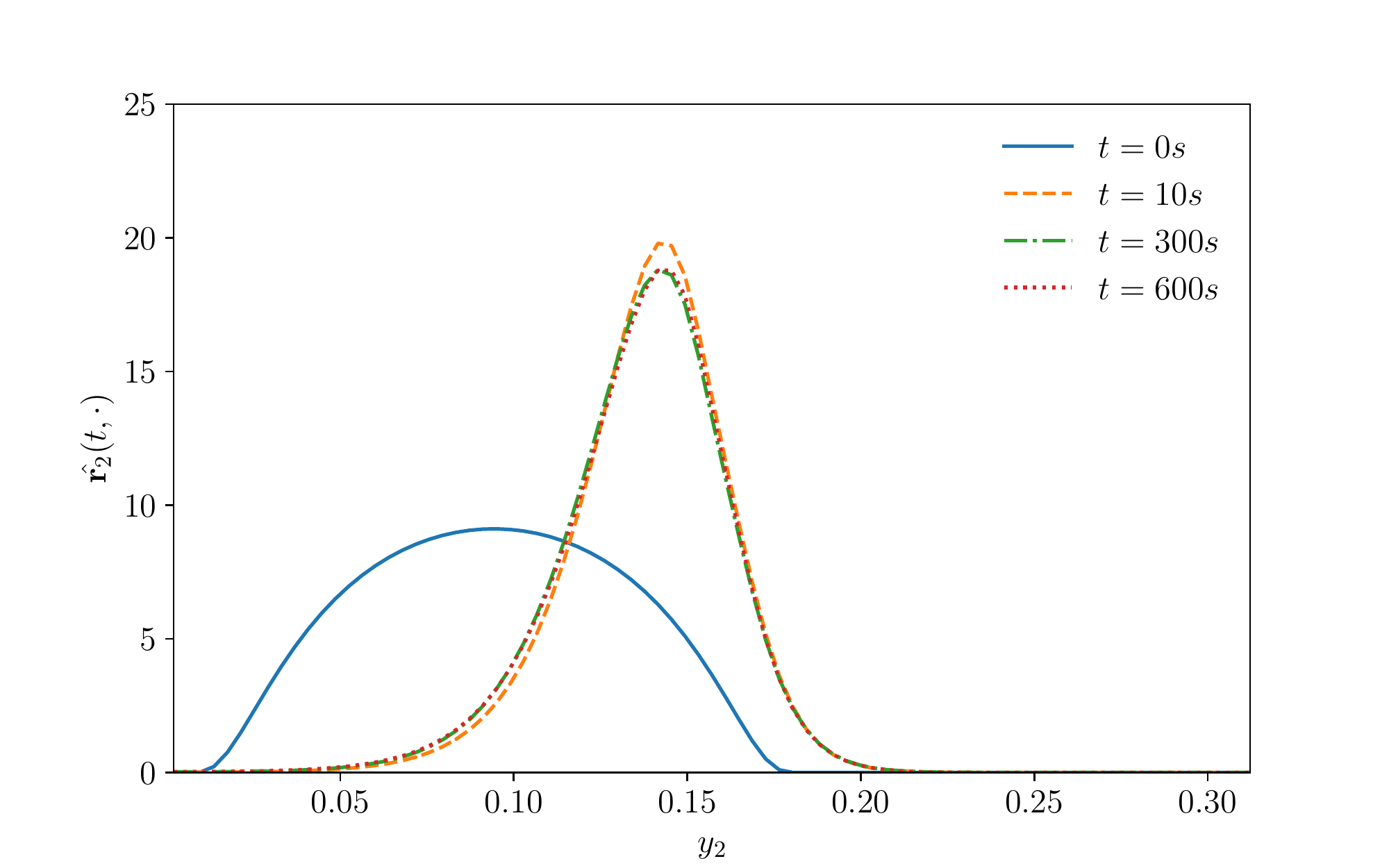}
	\caption{(top) The evolution of the weights of the approximate value function. As time progresses, the weights appear to approach an equilibrium. Note, however, we have not proven stability in this work, only existence and uniqueness. (bottom) The marginal distribution for the speed. The speeds are initially distributed according to a bump function, but later evolve into a much more concentrated distribution at a higher speed.}
\end{figure}
\begin{figure}
	\centering
	\includegraphics[scale=.6, trim={0 0 0 0}]{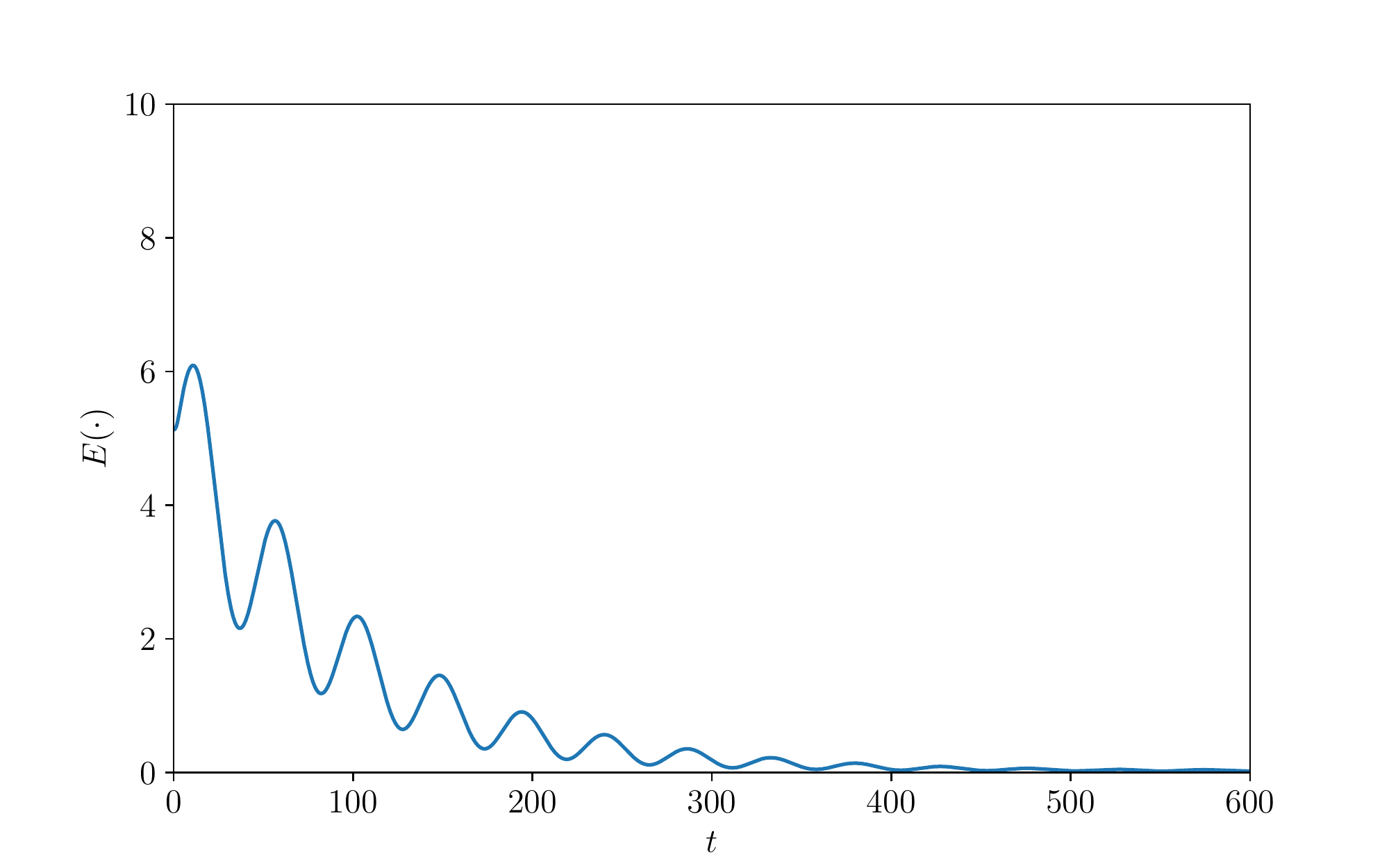}
	\caption{The evolution of the HJB-I error as time progresses. The error oscillates as the vehicles move around the track, and generally decreases over time, indicating that ADP process begins to learn the optimal control and disturbance.}
\end{figure}
\subsection{Model Parameters}
We made the following selections for model parameters. The length of the road was taken to be $\mathfrak L=2 \pi$. $\epsilon = .0005$, $s_{max} = \mathfrak L/20$, $u_{max} = s_{max}/6$,  $w_{max} = u_{max}/10$. $\gamma=10$. We took $\beta = 2$, and $T= 600$. We took $K=2$, so there are 6 (nonzero) basis functions for the value function approximation. $ \alpha =1$. $\phi$ is taken to be:
\[
\phi(y_1,\eta_1):= \sin(2 \pi \frac{y_1 - \eta_1}{\mathfrak L})
\]
and the initial (smooth) probability density is:
\[
\rho_0(y):= \frac{1}{C}e^{10 \cos(2 \pi \frac{y_1-\frac{L}{2}}{L})}\mathbb I_{|{y_{2}} -\frac{3}{10} s_{max}| < \frac{1}{11}}({y_{2}})e^{\frac{1}{|11({y_{2}} -\frac{3}{10} s_{max})|^2 - 1}}
\]
where $C$ is a normalizing constant. We took $\theta^{-1} = 10^{-2}$, and $A_0 = B_0 = \mathbf .1 \cdot \mathbf 1_{3 \times 3}$, where $\mathbf 1_{3 \times 3}$ is the matrix of all ones.
\subsection{Numerical Methods}
Let the numerical solution of the FK equation be piecewise constant, denoted by $\hat \rho: [0,T] \times Y \rightarrow \mathbb R$, and the numerical solution of the weights of the value function be piecewise constant $\hat A, \hat B: [0,T] \rightarrow \mathbb R^{K \times K}$.
For a detailed description of our numerical methods for the FK equation, please refer to our work in \cite{tirumalai2021robust}. In short, we discretize the FK equation in space using the finite volume method. We use the Rusanov numerical fluxes to approximate the hyperbolic part, and the second-order central difference method on the parabolic part. Both the numerical solution of the FK equation and the numerical solution of the value function weights were time-marched using the stability preserving second-order Runge-Kutta scheme \cite{shu1988efficient}. We took our timestep to be $\Delta t = .0025$, and took $81^2$ grid points in position-speed.
\subsection{Discussion of Numerical Results}
To begin this discussion, we first describe some macroscopic quantities related to the vehicle dynamics.
Define the spatial density $\mathbf r_1:[0,T] \times \mathbb T \rightarrow \mathbb R$, and the momentum density $\mathbf j:[0,T] \times \mathbb T \rightarrow \mathbb R$:
\[
\mathbf r_{1} (t,\cdot):= \int_\mathbb R \rho(t,\cdot) \text{ } dy_2,
\]
\[
\mathbf j(t,\cdot):= \int_{\mathbb R} y_2 \rho(t,\cdot) \text{ } dy_2  =: \mathbf r_1 (t,\cdot) \mathbf v(t,\cdot),
\]
under the assumption that $\rho(t,\cdot)$ is extended by $0$ for ${y_{2}} \notin [0,s_{max}]$.
We integrate the FK equation over $y_2 \in \mathbb R$, and we obtain:
\[
\partial_t \mathbf r_1 + \partial_{y_1} (\mathbf r_{1} \mathbf v ) = 0 \text{ in } (0,T] \times \mathbb T
\]
with $\mathbf r_1(t,y_1) =\mathbf r_1(t,y_1 + \mathfrak L)$, and initial conditions obtained from $\rho_0$.  We obtain spatial density $\hat{\mathbf r}_1(t,\cdot)$ and bulk velocity $\hat {\mathbf v}(\cdot,t)$ numerically from $\hat \rho(\cdot,t)$ via Riemann sums. These are plotted in Fig. 1. We also plot the numerical version of:
\[
\mathbf r_2 (t,\cdot):= \int_\mathbb R \rho(t,\cdot) \text{ } dy_1,
\]
in (bottom) of Figure 2. 
Similarly to our result in \cite{tirumalai2021robust}, initially, the bulk velocity $\hat{\mathbf v}$ forms a slowdown region. The vehicles ahead of the congestion spread into the sparsely occupied regions of the road, and as the slowdown region is then gradually evolved so that it dissipates. The speed distribution concentrates at a higher speed than the system was initialized at. As expected, the performance is not as good as the exact dynamic programming approach we took in our previous work \cite{tirumalai2021robust}, and the vehicles move more slowly. However, the computation of the control in our previous work was completely offline. Here, we obtained a control adaptively and on-line. The smoothing of the bulk velocity profile in the bottom of Fig. 1 is indicative of comfortable, fuel-efficient travel. This result indicates some applicability of this control method to more practical scenarios, such as agent-based traffic simulations with observation and estimation of the mean-field distribution.

The dynamics of the HJB-I loss as depicted in Fig. 5 are particularly interesting. At least in simulation, it does not strictly decrease as the system dynamics progresses. The simulation identifies a question which needs answering: where can we initialize the ADP for this system so that the HJB-I loss is strictly decreasing? And, is this even possible? Other initializations which we do not reproduce here seem to lead to divergence, so this is an important question to answer.

\section{Conclusion}
In this paper, we posed a robust discounted horizon mean-field game and obtained the stationary system which provides its solution. We used the HJB-I equation of this system to develop an ADP system, which consists of ODEs for the value function weights, and a forward Kolmogorov equation for the traffic density. We proved weak solutions to this system exist and are unique. Moreover, these controls are feasible for the mean-field game we pose. We implemented a numerical simulation, and obtained an adaptive control for the traffic density which dissipates traffic and increases and smooths the bulk velocity.

Of course, the question of  whether this approach results in solutions that converge to the solution of the stationary mean-field game as the number of basis functions $K \rightarrow \infty$ is also open. 

There are several more avenues which we intend to explore. These are:
\begin{enumerate}
	\item $\epsilon(N)$ suboptimality of the optimal control from the mean-field game for the finite-size game;
	\item Suboptimality bounds for the control from the ADP system for the mean-field game; and
	\item Addition of multiple lanes and collision avoidance.
\end{enumerate}
\noindent
We leave these questions to be answered in our future work.

\bibliographystyle{siam}  
\bibliography{references.bib}
\end{document}